\documentclass[a4paper,12pt]{article}

\usepackage{amsmath,amsthm,amsfonts,mathtools,amsthm,hyperref,dsfont, enumerate, verbatim, graphicx, epstopdf, enumitem, cancel, subcaption, amssymb, comment}
\newcounter{cond} % or any name you like

\makeatletter
\newcommand{\nameditem}[2]{%
  \def\thecond{#1}% how this condition will appear in refs
  \refstepcounter{cond}% step the "cond" counter so label works
  \item[#1]% how it appears in the list: (Regularity)
  \label{#2}% label to reference later
}
\makeatother

\usepackage{graphicx}
\usepackage[dvipsnames]{xcolor}
\usepackage{ulem}
\newtheorem{thm}{Theorem}
\newtheorem{lemma}[thm]{Lemma}
\newtheorem{prop}[thm]{Proposition}
\newtheorem{cor}[thm]{Corollary}
\newtheorem{rem}[thm]{Remark}

\newtheorem{ass}{Assumption}

\newcommand{\bfP}{\mathbf{P}}
\newcommand{\bfE}{\mathbf{E}}

\newcommand*{\norm}[1]{\lVert #1 \rVert}

\newcommand{\I}{\mathds{1}}

\newcommand{\ee}{\mathbf{e}}

\newcommand{\vp}{\varphi}

\newcommand{\eps}{\varepsilon}

\newcommand{\giv}{\, \big| \,}

\makeatletter

\title{Convergence to the Brownian CRT for critical branching Markov processes}

\author{Emma Horton\thanks{University of Warwick. E-mail: \texttt{emma.horton@warwick.ac.uk}}   
\and
Ellen Powell\thanks{Durham University. E-mail: \texttt{ellen.g.powell@durham.ac.uk}}}

\begin{document}
\maketitle
\begin{abstract}
We prove an invariance principle for a general class of continuous time critical branching processes with spatial motion and finite variance (non-local) branching mechanism. We show that the genealogical trees, viewed as random compact metric measure spaces, converge under {conditioning to be large and}  rescaling to the Brownian continuum random tree in the Gromov–Hausdorff-weak topology, establishing a universal scaling limit for these processes.

\medskip


\medskip

\end{abstract}

\section{Introduction}

\subsection{Branching Markov Processes}

Let $E$ be a Lusin space. Throughout, we write $L_\infty^+(E)$ for the space of non-negative bounded measurable functions on $E$ with norm $\norm{f}:=\mathrm{sup}_{x\in E}|f(x)|$, and $L^{+}_{\infty,1}(E)$ for the subset of functions in $L_\infty^+(E)$ which are uniformly bounded by one.

We consider a spatial branching process with finitely many particles in which, given their point of creation, particles evolve independently according to a Markov process, $(\xi, \bf P)$. We allow for the possibility that $(\xi, \mathbf P)$ undergoes (soft) killing at an exponential rate or (hard) killing on some absorbing set. As such we append a cemetery state $\{\dagger\}$ to $ E$, which is to be treated as an absorbing state, and regard $(\xi, \mathbf P)$ as conservative on the extended space $E\cup\{\dagger\}$, which is also a Lusin space. 

In an event which we refer to as `branching', particles positioned at $x$ die at rate $\gamma(x) \ge 0$ and instantaneously, new particles are created in $E$ according to a point process. The configurations of these offspring are described by the random counting measure
\begin{equation}
	\mathcal{Z}(A) = \sum_{i = 1}^N \delta_{x_i}( A), 
	\label{Z}
\end{equation}
for Borel $A$ in $E$. The law of the aforementioned point process may depend on $x$, the point of death of the parent, and we denote it by $\mathcal{P}_x$, $x\in E$, with associated expectation operator given by $\mathcal{E}_x$, $x\in E$.  This information is captured by the branching mechanism,
\begin{equation}
	\Phi(x,f) :=  \gamma(x)\mathcal{E}_x\left[\prod_{i = 1}^N f(x_i) - f(x)\right], \qquad x\in E, \, f\in L^+_{\infty,1}(E).
	\label{linearG}
\end{equation}
Without loss of generality we can assume that $\mathcal{P}_x(N =1) = 0$ for all $x\in E$ by viewing a branching event with one offspring as an extra jump in the motion. On the other hand, we do allow for the possibility that $\mathcal{P}_x(N =0)>0$ for some or all $x\in E$. We extend $\Phi$ to $E\cup\{\dagger\}$ by defining it to be zero on $\{\dagger\}$, i.e. particles cannot branch in the cemetery state. We assume throughout that $\gamma$ and $x \mapsto \mathcal E_x[N^2]$ are both in $L_\infty^+(E)$.
\medskip

The branching Markov process can be described via the process $X= (X_t, t\geq0)$ in the space of atomic measures on $E$ with non-negative integer total mass, denoted by $M(E)$, where
\[
X_t (\cdot) = \sum_{i =1}^{N_t}\delta_{x_i(t)}(\cdot), \qquad t\geq0.
\]
Above, $(N_t, t\ge 0)$ denotes the number of particles alive (note, particles in the cemetery state are \emph{not} considered alive) at each time $t\ge 0$ and their positions are given by $\{x_i(t)\, : \, i=1,\dots, N_t\}$. 
In particular, $X$ is Markovian in  $M(E)$. Its probabilities will be denoted $\mathbb{P}: = (\mathbb{P}_{\delta_x}, x\in E)$ where for $x\in E$, $\mathbb{P}_{\delta_x}$ denotes the law of $X$ initiated from $\delta_{x}\in M(E)$. The linear semigroup associated with $X$ is given by 
\begin{equation}
\label{E:Blinearsemigroup}
\psi_t[f](x):=\mathbb{E}_{\delta_x}[\langle f,  X_t \rangle], \qquad t\ge 0, \, x\in E, \, f\in L_\infty^+(E),
\end{equation}
where for $f \in L_\infty^+(E)$ and $\mu \in M(E)$, we set 
\[
  \langle f, \mu\rangle = \int_E f(y) \mu(dy).
\]

Finally, we say that $f$ is in the domain of the extended generator of a Markov process $(Z, P_z)$ if there exists a bounded function $Lf$ such that
\[
  f(Z_t) - f(z) - \int_0^t Lf(Z_s)ds, \quad t \ge 0,
\]
is a martingale.

We now introduce some assumptions on the branching mechanism and the single particle motion that will be used at various points throughout the rest of the article, starting with the finite variance assumption mentioned above.

\begin{ass}\label{A:main}
The most general assumptions we make on the Markov branching process are as follows. When we need only a subset of these assumptions, we will refer to them specifically.
	\begin{enumerate}[label={}]\setlength{\itemsep}{0em}
		%\item[{\bf ()}] The branching rate $\gamma(x)$ and the second moment $\mathcal{E}_x[N^2]$ are uniformly bounded above. \label{A:var}
        \nameditem{\bf (VAR)}{a:VAR} $\gamma$ and $x \mapsto \mathcal E_x[N^2]$ are both in $L_\infty^+(E)$.
        {\nameditem{\bf (CAD)}{a:CAD} The Markov process $(\xi, \mathbf P)$ is càdlàg.}
		\nameditem{\bf (PF)}{a:PF} There exists $\varphi\in L_\infty^+(E)$ and a probability measure $\tilde{\varphi}$ on $E$, normalised so that $\langle \tilde \varphi, \varphi \rangle =1$, such that $\varphi$ is uniformly bounded away from zero on each compactly embedded subset of $E$, and for all $g\in L_\infty^+(E)$ and $t \ge 0$, we have 
		\[
		\langle \psi_t[g], \tilde{\varphi}\rangle = \langle g, \tilde{\varphi}\rangle \text{ and } \psi_t[\varphi]=\varphi.		\]
		Moreover there exists $\varepsilon>0$ such that 
		\[
		\sup_{g\in L_{\infty, 1}^+(E)}\|\varphi^{-1}\psi_t[g]-\langle g, \tilde{\varphi}\rangle\| = O(e^{-\varepsilon t}) \text{ as } t\to \infty.
		\]
		\nameditem{\bf (EXT)}{a:EXT} For all $x\in E$ the MBP becomes extinct $\mathbb{P}_{\delta_x}$-almost surely; that is, if $\zeta:=\inf \{t\, : \, N_t=0\}$, then
		\[
		\mathbb{P}_{\delta_x}(\zeta <\infty) = 1.\]
		\nameditem{\bf (ERG)}{a:ERG} There exist constants $C,M\in (0,\infty)$ such that for all $g\in L_\infty^+(E)$, setting $ \mathcal{V}_M[g](x):= \mathcal{E}_x [ \sum\nolimits_{ i\ne j} g(x_i)g(x_j) \mathbf{1}_{\{N\le M\}}] $,
		\[ \langle \gamma \mathcal{V}_M[g], \tilde{\varphi}\rangle \ge C \langle g, \tilde{\varphi} \rangle^2.\] \vspace{-.5cm}
		\nameditem{\bf (CC)}{a:CC} {The function $\varphi^2$ is in the domain of the extended generator of $(\xi, \mathbf P)$}, where we set $\varphi(\dagger)=0$ so that $\varphi,\varphi^2$ are well defined functions of $\xi$ for all time.  
	\end{enumerate}
\end{ass}

{Let us now comment on the above assumptions. First note that \ref{a:VAR} ensures that branching events do not result in excessively large fluctuations, so that we remain in the Brownian scaling regime. Next, \ref{a:CAD} simply allows us to make use of standard martingale results. The third assumption stipulates that the MBP is critical. Indeed, a more general version of \ref{a:PF} would give the existence of a triple $(\lambda, \varphi, \tilde\varphi)$ with $\lambda \in \mathbb R$ such that for all $t \ge 0$, 
\[
	\langle \psi_t[g], \tilde{\varphi}\rangle = {\rm e}^{\lambda t}\langle g, \tilde{\varphi}\rangle \text{ and } \psi_t[\varphi]={\rm e}^{\lambda t}\varphi,
\]
and
\[
	\sup_{g\in L_{\infty, 1}^+(E)}\|{\rm e}^{-\lambda t}\varphi^{-1}\psi_t[g]-\langle g, \tilde{\varphi}\rangle\| = O(e^{-\varepsilon t}), \quad t \to\infty.
\] 
The case $\lambda > 0$ corresponds to the supercritical case, $\lambda < 0$ the subcritical case, and $\lambda = 0$ the critical case. The (right) eigenfunction $\varphi$ can be thought of as an importance function that assigns a `weight' to each $x \in E$ according to how much mass a process initiated from $x$ contributes (on average) to the total mass of the system. On the other hand, the (left) eigenmeasure $\tilde\varphi$ describes the long-term spatial distribution of particles in the system. In essence, this assumption can be thought of as a functional version of the Perron-Frobenius decomposition for matrices but in the setting of semigroups. Now, in the critical regime, one would usually expect the branching process to become extinct almost surely in finite time however, it seems difficult to prove this assumption in full generality, hence the need for assumption \ref{a:EXT}. Assumption \ref{a:ERG} can be seen as an ergodicity assumption on the offspring distribution: on the event that the number of offspring is at most $M$, when distributed according to the stationary distribution, the offspring are at least as `spread out' as a process that produces two children whose positions are independently distributed. Note that these two assumptions also appear in the literature, see e.g. \cite{Yaglom2022, BMPI}. Finally, assumption {\ref{a:CC}} will allow us to work with the so-called Carré du Champ operator; we refer the reader to Section \ref{subsec:CC} and \cite{cdrg} for further details.

\begin{rem}Our framework is tailored to spatial branching processes where the motion of the particles affects the genealogy and causes it to be critical, as stipulated by assumption \ref{a:PF}. However, it may be the case that (some aspect of) the spatial motion is not relevant to the genealogy, and asking for Assumption \ref{A:main} to hold for the MBP viewed as a process on the whole state space $E$ is too restrictive. For example, imagine the (extreme) situation where we have a critical continuous time Galton Watson process on top of which particles perform some independent spatial motion during their lifetimes. In such a situation, one should project the process onto a smaller state space that captures only the genealogically relevant information, and where assumption \ref{a:PF} holds, in order to apply our results. Thus, while our results are formulated in a spatial setting, they do unify this with non-spatial scenarios. 
\end{rem}

We finish this subsection by discussing some concrete examples that satisfy the above assumptions. First, note that this result extends \cite{ellen} to non-local branching diffusions. For example, consider the case where $E$ is a compact subset of $\mathbb R^d$ with $C^2$ boundary, $\xi$ is a diffusion in $E$ and the offspring distribution satisfies 
\[
 \mathcal E_x[\mathcal Z[f]] = \int_E f(y) K(x, {\rm d}y), \quad x \in E, f \in L_\infty^+(E).
\]
Assume that the resulting branching process is critical and the diffusion $\xi$ is sufficiently regular. Then, in the case that $K(x, {\rm d}y) = (1-\varepsilon)\delta_x + \varepsilon \nu({\rm d}y)$, where $\nu$ has full support, or $K(x, {\rm d}y) = k(x, y){\rm d}y$ with $k(x, y) \ge c > 0$, for example, our assumptions hold.

Another example is that of a branching piecewise deterministic Markov process (PDMP). In this case, particles carry a spatial position, $r \in D \subset \mathbb R^d$, and a velocity $v \in V \subset \mathbb R^d$. Between branching events, particles move according to a PDMP, $\xi$, switching velocity at a bounded rate $\lambda(\cdot)$ and sampling its outgoing velocity according to $q(\cdot, {\rm d}y)$, say. Assume that $E = D \times V$ is bounded, particles are killed on exiting $D$, $\lambda q$ is uniformly bounded above and away from $0$, there is a uniformly positive probability of producing at least two offspring at a branching event and the branching process is critical. Then the above assumptions also hold. We refer the reader to \cite{bNTEbook} for a proof of \ref{a:PF} and \ref{a:EXT}. Proving \ref{a:ERG} is straightforward. For \ref{a:CC} we refer the reader to \cite[Sections 2 and 4]{cdrg}. 
} 

{Finally, we consider the case of a multi-type continuous time Galton Watson process where particles carry type $i \in E \subset \mathbb N$. In this case, checking the above conditions reduces to studying the 
mean reproduction matrix $M = (m_{i,j})_{i, j \in E}$, where $m_{i,j} = \mathcal E_i[\mathcal Z[\mathbf 1_{j}]]$. In the case that $E$ is finite, \ref{a:VAR} holds and $M$ is irreducible, the standard Perron Frobenius theorem for matrices yields the existence of left and right eigenvectors for $M$ associated with the spectral radius of $M$, $\rho(M)$. In the case that $\rho(M) = 1$, \ref{a:PF} and \ref{a:EXT} are satisfied. 
}

\subsection{Main result}

Given a critical MBP $X$ we let $T$ denote the associated genealogical tree. We denote by ${\bf r}$ the root of the tree, $d$ the genealogical distance and $\nu$ the total occupation time of the depth-first exploration of $T$, without backtracking. See Section \ref{subsec:mmspace} for rigorous definitions of these objects. We then define
\[
  \mathcal T_{n, x} := (T, \frac1n d, \frac{1}{n^2}\nu, {\bf r}), \quad n \ge 0, \, x \in E,
\]
to be the associated random pointed metric measure spaces, under the law 
\[
\mathbb P_{\delta_x}( \cdot | N_n > 0)
\]
(where the maximum distance $d(\cdot,\cdot)$ is of order $n$).

We also let $\mathcal T_{\bf e}$ denote the Brownian CRT generated from a Brownian excursion with speed $\sigma^2(f)$ (a constant defined in \eqref{eq:sf}\footnote{ note that the speed of the Brownian excursion does not affect the metric space structure of $\mathcal{T}_{\bf e}$ but it does affect the measure}) conditioned to have height at least $1$ and viewed as a pointed metric measure space with its natural measure and root (see Section \ref{subsec:mmspace}).

\begin{thm}\label{thm:main}
Under Assumption \ref{A:main} %\ref{a:VAR}, \ref{a:CAD}, \ref{a:PF}, \ref{a:EXT}, \ref{a:ERG}, \ref{a:CC}, \ref{a:HK}, 
we have
\begin{equation}\label{eq:main}
\mathcal T_{n, x}  \overset{d.}{\to} \mathcal T_{\bf e}, \quad n \to \infty,
\end{equation}
with respect to the Gromov-Hausdorff-weak topology. 
\end{thm}

As discussed below, scaling limits of genealogies of critical branching process have been investigated in various settings over the past several decades. Our result proves convergence to the CRT for critical branching process with spatial motion and non-local branching in a very general setting. Moreover, our proof employs a new, largely non-technical, approach. {This approach proceeds through a depth-first exploration, introduced in Section \ref{subsec:labels}, and hence uses an auxiliary measurable ordering of the offspring to label particles and encode the genealogy. The conclusion of the above result, however, remains independent of this ordering.}

{An extension of Theorem \ref{thm:main} to discrete time branching Markov processes should also hold under analogous assumptions. In this setting, each particle at generation $n \ge 0$ carries a type $x \in E$ and produces offspring according to the point measure $(\mathcal Z, \mathcal P_x)$. The resulting genealogical tree is equipped with the graph distance, and the occupation measure is replaced by the counting measure on vertices. The depth-first exploration indexed by tree length would be replaced by the lexicographic exploration of vertices, and the martingale introduced in section \ref{sec:martingale} and its predictable quadratic variation admit natural discrete time analogues. Under discrete time analogues of the above assumptions and preliminary results introduced in sections \ref{subsec:spine} and \ref{subsec:results} (we refer the reader to \cite{discretemoments} and in particular, its supplementary material), we expect the proof strategy to closely follow that of the continuous time case. Such a result would complement the results obtained in \cite{felix}.
}

\subsection{Literature}
Genealogical scaling limits for finite variance critical branching processes were first considered by Aldous \cite{CRT1, CRT3}, who showed that if you condition a critical Galton-Watson process to have a large total progeny then, under the assumption that the offspring distribution has finite variance, the rescaled genealogical tree converges (in the sense of its encoding contour function) to the Brownian CRT. This result has since been extended to various other settings including Galton-Watson trees with infinite variance offspring distribution \cite{LeGall02}, multitype Galton-Watson processes with a finite number of types \cite{miermont08}, and branching diffusions in bounded domains \cite{ellen}, to name just a few.

{More recently, Foutel-Rodier obtained a general and technically powerful result for discrete-time branching Markov processes, assuming finite positive moments of all orders for the offspring distribution \cite{felix}. Our work can be viewed as complementary to this framework, as well as an extension in two directions: we show that only second-moment assumptions on the offspring distribution are sufficient, and we treat the continuous-time setting. The proof also follows a different philosophy, relying on new analytic and probabilistic ideas -  most notably a robust martingale argument - rather than the combinatorial techniques and moment computations that underpin earlier approaches.}

Finally, let us mention that since we view our genealogical tree 
as a random metric measure space (mm-space),
we will rely on several modern results concerning  
 topologies on mm-spaces  
and convergence of random mm-spaces, 
\cite{mm-space, mmm-space, felixmmspace}. We refer the reader to section \ref{subsec:mmspace} for a summary and brief discussion of the results that will be needed for the proof of Theorem \ref{thm:main}.

\paragraph{Outline of the paper.} The rest of this paper is set out as follows. In the next section we present various preliminary results that will be used to prove the main theorem. In particular, section \ref{subsec:spine} discusses various many-to-one representations of MBPs and associated ergodicity results, section \ref{subsec:results} houses existing results in the literature that will be used to prove Theorem \ref{thm:main}, in section \ref{subsec:CC} we give a brief introduction to the Carré du Champ operator which will be necessary to define various martingales for our proofs, section \ref{subsec:labels} formalises notation required to make sense of the genealogies of the MBP and section \ref{subsec:mmspace} discusses relevant theory of metric measure spaces. 
Section \ref{sec:proof} is then concerned with the proof of Theorem \ref{thm:main}. {We start, in Section \ref{sec:martingale}, by identifying a key martingale associated with the depth-first exploration of a critical MBP genealogical tree, generalising the Lukasiewicz path associated with a Galton-Watson tree, and serving as a proxy for the height (or contour) function of the tree. In section \ref{sec:fclt} we establish a functional central limit theorem for this martingale, which provides the crucial connection to the Brownian CRT. Finally in section \ref{sec:mainproof}, we relate the martingale precisely to the height function of a critical MBP genealogical tree conditioned to be large, which allows us to complete the proof.}

\section{Preliminaries and existing results}

\subsection{Spine decomposition}\label{subsec:spine}

In this section, we present some existing results pertaining to various single particle motions associated with a MBP. We refer the reader to \cite{Yaglom2022} for the proofs of these results. 

For $f\in L_\infty^+(E)$ and $x\in E$, define
\begin{equation}
m[f](x):=\mathcal{E}_x[\langle f, \mathcal{Z} \rangle] = \mathcal{E}_x\bigg[\sum_{i=1}^N f(x_i)\bigg].
\label{E:m}
\end{equation}
We will also write $m(x):=m[\mathbf{1}](x)=\mathcal{E}_x[N]$ for the mean of the offspring distribution, which is uniformly bounded from above thanks to \ref{a:VAR}.

Another way to describe the semigroup $\psi_t$ is via the so-called many-to-one formula. In this setting, due to the possibility of non-local branching, we need to introduce another Markov process $Y=(Y_t, \, t\ge 0)$, which evolves according to the dynamics of $(\xi, \mathbf P)$ except that at instantaneous rate $\gamma(x)m(x)$, it jumps from its current position $x$ to a new random position, lying in $A\subset E$ with probability $m^{-1}(x) m[\mathbf 1_A](x)$. Let $\tilde{\bfP}=(\tilde{\bfP}_x, \, x\in E)$ denote its law and $\tilde{\bfE}=(\tilde{\bfE}_x, \, x\in E)$ its associated expectation. Then by conditioning on the first jump of the process $Y$, it is straightforward to deduce the following many-to-one formula. 

\begin{lemma}[Many-to-one formula] 
	\label{L:mt1}
 	Under \ref{a:VAR} (in fact, \ref{a:VAR} can be relaxed to a first moment assumption), for each $x \in E, t\ge 0$ and $f\in L_\infty^+(E)$, we have
	\[ 
	\psi_t[f](x)=\tilde{\bfE}_x[e^{\int_0^t \gamma(Y_s)(m(Y_s)-1) \mathrm{d}s} f(Y_t)\mathbf 1_{\{t < \mathtt k\}}],
	\]
	where $\mathtt k := \inf\{t \ge 0 : Y_t \notin E\}$.
\end{lemma}

\bigskip

Let us now introduce a useful result known as the spine decomposition. We assume that assumptions \ref{a:CAD} and \ref{a:PF} hold and, for an initial configuration $\mu\in M(E)$, define
\begin{equation}
	\label{E:W}
	W_t:=\frac{\langle \varphi, X_t \rangle}{\langle \varphi, \mu \rangle}, \qquad x\in E, \, t\ge 0.
\end{equation}

Then, due to \ref{a:PF} and the branching Markov property, $(W_t, \, t\ge 0)$ is a mean-one martingale under $\mathbb{P}_\mu$. This means we can define a new measure $\mathbb{P}_\mu^\varphi$ on the Markov branching process via
\begin{equation}
	\label{E:COM}
	\frac{\mathrm{d}\mathbb{P}_\mu^\varphi}{\mathrm{d} \mathbb{P}_\mu} \Big|_{\mathcal{F}_t} = W_t, \qquad t\ge 0.
\end{equation}
It is well-known that under the new measure, we have a so-called \textit{spine decomposition} for the branching process. We refer the reader to \cite[Chapter 11]{bNTEbook} for the proof and further details.

\begin{prop}[Spine decomposition]\label{prop:spine}
Under  assumptions \ref{a:CAD}, \ref{a:PF} and \ref{a:VAR}	(in fact, \ref{a:VAR} can be relaxed to a finite first moment assumption), the branching process $X$ under $\mathbb{P}^\varphi$ can be constructed as follows. 
	\begin{enumerate}
		\item From the initial configuration $\mu = \sum\nolimits_{j = 1}^n \delta_{x_i}\in {M}(E)$ with an arbitrary enumeration of particles, the $i^*$-th individual is selected and marked \textit{spine}, so that, given $\mu$, 
		the probability that $i^*=j$ is given by $\varphi(x_j)/\langle\varphi,  \mu\rangle$ for $j=1,\dots, n$.
		\item The individuals $j\ne i^*$ in the initial configuration each issue independent copies of $(X,\mathbb{P}_{\delta_{x_j}})$ respectively. 
		\item The marked individual, ``spine'', issues a single particle whose motion has the law of the many-to-one motion $(Y,\tilde{\bfP})$ weighted by 
        \[ e^{\int_0^t \gamma(Y_s)(\frac{m[\varphi](Y_s)}{\varphi(Y_s)}-1) \mathrm{d}s}\frac{\varphi(Y_t)}{\varphi(x)} .\]
		\item The spine undergoes branching at the accelerated rate 
		\[ \rho(x):=\gamma(x) \frac{m[\varphi](x)}{\varphi(x)} \]
		when at $x\in E$, at which point, it produces a random number of particles according to the random measure $\mathcal{Z}$ on $E$ with law $\mathcal{P}_x^\varphi$, where 
		\[
		\frac{\mathrm{d}\mathcal{P}_x^\varphi}{\mathrm{d} \mathcal{P}_x} = \frac{\langle \varphi, \mathcal{Z} \rangle}{m[\varphi](x)}.
		\]
		\item Given $\mathcal{Z}$ from the previous step, $\mu$ is redefined as $\mu=\mathcal{Z}$ and Step 1 is repeated.
	\end{enumerate}
\begin{rem}\label{rem:L}
In the above description, it is implicit that for $Y$ be the many-to-one motion (with law $\tilde{\bfP}$); the process 
         \[ e^{\int_0^t \gamma(Y_s)(\frac{m[\varphi](Y_s)}{\varphi(Y_s)}-1) \mathrm{d}s}{\varphi(Y_t)}\]
         is a martingale (see \cite[Chapter 11]{bNTEbook} for a proof). This implies that
         $
         \varphi(Y_t)-\int_0^t \gamma(Y_s)(m[\varphi](Y_s)-\varphi(Y_s))
         $
         is a $(Y,\tilde{\bfP})$ martingale, and in turn that
        $
         \varphi(\xi_t)-\int_0^t \gamma(\xi_s)(m(\xi_s)-1)\varphi(\xi_s)
         $
         is a $(\xi,\bfP)$ martingale. In particular, $\varphi$ is in the domain of the extended generator for the Markov process $(\xi,\bfP)$.
\end{rem}

The above implies that the motion of the particle marked spine (including the jumps that it experiences at branching events) has law $\tilde{\bfP}^\varphi$ where 
\[ 
\frac{\mathrm{d}\tilde{\bfP}^\varphi_x}{\mathrm{d}\tilde{\bfP}_x} = e^{\int_0^t \gamma(Y_s)(m(Y_s)-1) \mathrm{d}s}\frac{\varphi(Y_t)}{\varphi(x)}\mathbf{1}_{\{t<\mathrm{k}\}}, \qquad t\ge 0,\, x\in E,
\]
where $\mathrm{k}:=\inf\{t>0 \, : \, Y_t\notin E\}$. From this we see that the process $Y$ under $\tilde{\bfP}^{\varphi}$ is conservative, and satisfies 
\[ \tilde{\bfE}_x^{\varphi}[f(Y_t)]=\frac{1}{\varphi(x)} \psi_t[\varphi f], \qquad t\ge 0,\, f\in L_\infty^+(E),\]
with stationary distribution 
\[ \varphi(x)\tilde{\varphi}(\mathrm{d}x) , \qquad x\in E.\]
\end{prop}

\bigskip

\bigskip

Now we state a result regarding the ergodicity of the spine. { The proof follows, for example, by \cite[Theorem 5.1]{Yaglom2022}: computing first and second moments and using the exponential rate of convergence to stationarity for the spine motion, as assumed in \ref{a:PF}}. 

{\begin{lemma}\label{lem:ergodic}
Suppose that  $F:E\to [0,\infty)$ satisfies $\sup_{x \in E}|\varphi(x)F(x)| < \infty$. Then 
\[
  \lim_{t \to \infty}\sup_{x \in E} \left| \tilde{\mathbf E}_x^\varphi \left[\int_0^1 F(Y_{ut})du\right] - \langle \varphi \tilde\varphi, F\rangle \right| = 0
\]
and 
\[
\frac1t \int_0^t F(Y_u) du \to \langle \varphi\tilde\varphi,F \rangle 
\quad \tilde{\mathbf{P}}_x^\varphi\text{- almost surely as } t\to \infty.\]
\end{lemma}}

\subsection{Limit theorems for critical MBPs}\label{subsec:results}
We now present some further existing results for our class of MBPs that will be used throughout the proof of Theorem \ref{thm:main}. Theorem \ref{thm:survival} concerns the classical Kolmogorov survival probability for critical branching processes and Theorem \ref{thm:Yaglom} pertains to the so-called Yaglom limit; the proofs of these results can be found in \cite{BMPI}. Corollary \ref{cor:joint_Yaglom} and Proposition \ref{prop:Q-process}
are consequences of the aforementioned Yaglom limit. Finally, Theorem \ref{thm:moments} characterises the asymptotic behaviour of the moments of $X$ and its occupation, and we refer the reader to \cite{bmoments, moment_correction} for its proof.

\begin{thm}[Kolmogorov survival probability]\label{thm:survival}
Suppose \ref{a:PF},\ref{a:EXT} and \ref{a:ERG} hold. Then, for all $\mu\in M(E)$,
\begin{equation*}
 \lim_{t \to\infty} t\mathbb{P}_{\mu} (N_t > 0) =  \frac{2\langle \varphi, \mu\rangle }{\langle {\gamma}\mathcal{V}[\varphi],\tilde\varphi\rangle},
 \label{eq:Kolmogorov}
 \end{equation*}
 where $\mathcal V[g](x) := \lim_{M \to \infty}\mathcal V_M[g](x) = \mathcal E_x[\mathcal Z[g]^2] - \mathcal E_x[\mathcal Z[g]]^2$.
\end{thm}

\medskip

\begin{thm}[Yaglom limit]\label{thm:Yaglom}
Suppose \ref{a:PF}, \ref{a:EXT} and \ref{a:ERG} hold.
Then, for all $\mu\in M(E)$,
\begin{equation*}
\lim_{t \to \infty}\mathbb E_{\mu}\left[ {\rm exp}\left(-\theta\frac{\langle f, X_t\rangle}{t} \right) \Big|\, N_t >0 \right] = \frac{1}{1 +  \tfrac12\theta \langle  f, \tilde\varphi\rangle \langle {\gamma} \mathcal V[\varphi], \tilde\varphi \rangle},
\label{eq:Yaglom}
\end{equation*}
where $\theta\geq 0$ and $f\in L_\infty^+(E)$.
\end{thm}

\medskip

\begin{cor}\label{cor:joint_Yaglom}
Assume the setting of Theorem \ref{thm:Yaglom}. For $f_1, f_2 \in L_\infty^+(E)$, under $\mathbb P_{\delta_x}(\cdot \mid N_t > 0)$, we have the joint convergence
\[
  \left(\frac{\langle f_1, X_t\rangle}{t}, \frac{\langle f_2, X_t\rangle}{t} \right) \overset{d.}{\to} (Y(f_1), Y(f_2)),
\]
where $Y(f_1)$ and $Y(f_2)$ denote the same exponential random variable but with rates $\frac12 \langle  f_1, \tilde\varphi\rangle \langle  \mathcal V[\varphi], \tilde\varphi \rangle$ and $\frac12 \langle  f_2, \tilde\varphi\rangle \langle  \mathcal V[\varphi], \tilde\varphi \rangle$, respectively. 
\end{cor}

The proof of this result follows easily from the proof of Theorem \ref{thm:Yaglom} in \cite{BMPI}. Indeed, the first (and main) step of the proof of Theorem \ref{thm:Yaglom} is to show that the result holds for $f = \varphi$. The result then follows for general $f$ by writing $f = \langle f, \tilde\varphi\rangle f + (f - \langle f, \tilde\varphi\rangle f)$ and showing that, conditional on survival up to time $t$, $\langle f-\langle f, \tilde\varphi\rangle f, X_t \rangle / t$ converges to zero. Applying this argument to both $f_1$ and $f_2$ proves the above corollary. 

\medskip

{Now we state a result that is well-known but is not explicitly stated in the literature in a fully general setting. The proof of this result follows verbatim the proof of \cite[Proposition 1.5]{ellen}.}

\begin{prop}\label{prop:Q-process}
    Assume the setting of Theorem \ref{thm:survival}. Let $R > 0$, $A \in \mathcal F_R$ and $x \in E$. Then
    \[
      \lim_{t \to \infty}\mathbb P_{\delta_x}(A \giv N_t > 0) = \mathbb P_{\delta_x}^\varphi(A).
    \]
\end{prop}

\bigskip

For $k \ge 0$, $t \ge 0$, $x \in E$ and $f \in L_\infty^+(E)$, define
\[
  \psi_t^{(k)}[f](x) = \mathbb E_{\delta_x}\left[\langle f, X_t\rangle^k\right]\quad
\text{ and } \quad 
  \eta_t^{(k)}[f](x) = \mathbb E_{\delta_x}\Big[\int_0^t \langle f, X_s\rangle{\rm d}s\Big].
\]

\begin{thm}\label{thm:moments}
Suppose that \ref{a:PF} holds and that for some $k \ge 2$, $\sup_{x \in E}\mathcal E_x[N^k] < \infty$. For $\ell \le k$ and $t \ge 0$, set
\[
  \Delta_t^{(\ell)} := \sup_{x \in E, f \in L_\infty^+(E)}\left|t^{-(\ell - 1)}\varphi(x)^{-1}\psi_t^{(\ell)}[f](x) - 2^{(\ell - 1)}\ell! \langle f, \tilde\varphi\rangle^{\ell}\langle \mathcal V[\varphi], \tilde\varphi\rangle^{\ell - 1} \right|,
\]
and 
\[
  \Theta_t^{(\ell)} := \sup_{x \in E, f \in L_\infty^+(E)}\left|t^{-(2\ell - 1)}\varphi(x)^{-1}\eta_t^{(\ell)}[f](x) - 2^{(\ell - 1)}\ell! \langle f, \tilde\varphi\rangle^{\ell}\langle \mathcal V[\varphi], \tilde\varphi\rangle^{\ell - 1}L_\ell \right|,
\]
where $L_1 = 1$ and for $\ell \ge 2$, $L_\ell$ is defined through the recursion $L_\ell = \frac{1}{2\ell - 1}\sum_{i = 1}^{\ell - 1}L_i L_{\ell - i}$.

\smallskip

Then, for all $\ell \le k$, 
\begin{align*}
\sup_{t \ge 0}\Delta_t^{(\ell)} < \infty, \qquad \lim_{t \to \infty}\Delta_t^{(\ell)}  = 0, \qquad
\sup_{t \ge 0}\Theta_t^{(\ell)} < \infty, \quad \text{ and } \quad \lim_{t \to \infty}\Theta_t^{(\ell)}  = 0.
\end{align*}
\end{thm}

\subsection{A functional central limit theorem}

For ease of reference, we present a result that was stated and proved in \cite[Theorem 3.12]{cdrg}, which we will use in the proof of Theorem \ref{thm:main}.

\begin{thm}[Functional Central Limit Theorem] \label{thm:CLT}
	Let $(Z_N(t))_{t\ge 0}$ denote a sequence of càdlàg processes, which may be defined on diﬀerent probability spaces. Suppose that
 $t\mapsto Z_N(t)-Z_N(0)$ are local square integrable martingales, and assume moreover
that:
\begin{itemize}
	\item[(a)]$Z_N(0)$  converge in distribution as $N\to \infty$;
	\item[(b)] the following limit holds, $\lim_{N\to \infty} \mathbb{E}[\sup_{[0,T]}|\Delta Z^N(t)|^2]=0$;
	\item[(c)] for each $N$ the predictable quadratic variation $\langle Z^N \rangle_t$ is continuous; and 
	\item[(d)] there exists a continuous increasing deterministic function $t\mapsto v(t)$ such that, for all $t\in [0,T]$, 
	\[ \langle Z^N\rangle_t \overset{.}{\rightarrow} v(t) \text{ as } N\to \infty.
    \]
\end{itemize}
Then $(Z_N(t))_{t\in [0,T]}$ converges in law (under the Skorokhod topology) towards
$(Z(t))_{t\in [0,T]}$, where $(Z(t)-Z(0))_{t\in [0,T]}$ is a Gaussian process, independent of $Z(0)$, with independent increments and variance function $v(t)$.
\end{thm}

\subsection{Carré du Champ operator}\label{subsec:CC}
The preliminaries in this subsection will be used in section \ref{sec:martingale} to identify our key martingale and compute its quadratic variation. We assume that \ref{a:VAR}, \ref{a:PF} and \ref{a:ERG} hold. Recalling $\varphi$ from assumption \ref{a:PF}, we set
\[
 \Gamma(\varphi) = L\varphi^2 - 2\varphi L\varphi.
\]
{As discussed after the statement of Assumption \ref{A:main}, \ref{a:VAR} and \ref{a:PF} imply that $\varphi$ is in the domain of the extended generator of $(\xi, \bfP)$ so that, in particular, } 
\[
  C_t := \varphi(\xi_t) - \varphi(x) - \int_0^tL\varphi(\xi_s) ds
\]
is a martingale. Moreover, as stated in \cite[Section 4.3]{cdrg}, its predictable quadratic variation is given by 
\begin{equation}\label{eq:Gamma}
  t \mapsto \int_0^t \Gamma(\varphi)(\xi_s) ds = \langle C \rangle_t, \qquad t\ge 0.
\end{equation}

Moreover, if we set $\tilde{C}_t=e^{-\int_0^t \gamma(\xi_s) ds}C_t+\int_0^t \gamma(\xi_s)e^{-\int_0^s \gamma(\xi_s) ds} C_s$ then 
$$d\tilde{C}_t=-\gamma(\xi_t)C_t e^{-\int_0^t \gamma(\xi_s) ds} dt+e^{-\int_0^t \gamma(\xi_s) ds} dC_t +\gamma(\xi_t)C_t e^{-\int_0^t \gamma(\xi_s) ds} dt = e^{-\int_0^t \gamma(\xi_s) ds} dC_t,$$ 
so $\tilde{C}_t$ is a is a càdlàg local martingale which is deterministically bounded for $t\in [0,T]$ for any $T > 0$. In particular, {conditionally on $(\xi_s)_{s\ge 0}$ let  $\sigma$ be a exponential random variable with rate $\gamma(\xi)$, that is, the probability that $\sigma$ exceeds $t$ is $\exp(-\int_0^t \gamma(\xi_s)) ds$. Then we have that $\bfE_x[C_{t\wedge \sigma}]=\bfE_x[\tilde{C}_t]$ and $\bfE_x[\tilde{C}_t]=0$ so 
\begin{equation}\label{eq:stoppedCexp} 
	\bfE_x[C_{t\wedge \sigma}]=0.
\end{equation}}

\subsection{Labelling}\label{subsec:labels}
In order to describe the genealogical structure of our branching Markov process, we introduce the set of Ulam-Harris labels, that we will associate to particles in the MBP in the standard manner:
\[ \Omega:=\{\emptyset\}\cup \bigcup_{n\ge 1} \mathbb{N}^n. \]
{To define these labels, at a branching event where particle with label $v$ has $N$ offspring, the labels $v1,v2,\dots vN$ are assigned to the offspring particles according to some ordering which is measurable with respect to the particle positions.} For $v, w \in \Omega$, we write $v \preceq w$ to mean that $v$ is an ancestor of $w$, which means there exists $u \in \Omega$ such that $vu = w$. {Moreover, we write $v \prec w$ to mean that $v \preceq w$ in the strict sense, that is, the possibility that $v=w$ is excluded.} We say that $v, w \in \Omega$ are siblings if there exists $u \in \Omega$ and $i \neq j$ such that $v = ui$ and $w = uj$. 

Given $v \in \Omega$, let $b_v$ denote the birth time of particle $v$ and $d_v$ denote the death time (with $b_v=d_v=\infty$ if label $v$ does not correspond to any particle in the MBP), and let $X_v(t)$ denote its position at time $t$ during its lifetime. {If a particle in the system is sent to $\dagger$ before a branching, its death time is the time it is sent to $\dagger$}. For each $v \in \Omega$, let $N_v$ denote the number of offspring produced by $X_v(d_v)$. Set $\mathcal{N}_t := \{v \in \Omega : t \in [b_v, d_v)\}$ so that $N_t = |\mathcal{N}_t|$.

\paragraph{Depth-first exploration.} Now let $\preceq_{\rm lex}$ denote the lexicographical ordering of vertices and set $\hat\Omega := \{v \in \Omega : b_v < \infty\}$. Then, for $v \in \hat\Omega$, set
\[
\sigma_v := {\sum_{w \prec_{\rm lex}v}}(d_w - b_w), 
\quad \text{ and } \quad 
  \tau_v := \sigma_v + (d_v - b_v).
\]

For $t \ge 0$, there exists a unique $v = v_t$ such that $t \in [\sigma_v, \tau_v)$. With this notation, let $h_t = b_{v_t} + (t-\sigma_{v_t})$ and $\zeta_t = X_{v_t}(h_t)$. Then, the depth-first exploration at time $t$ is given by $v_t$. In words, we move around the genealogical tree at speed one in a depth first order (without backtracking), and in this exploration, $v_t$ is the label of the particle being visited at time $t$, $h_t$ is its `height' in the tree (i.e.~the natural time in the branching process), and $\zeta_t$ is its position. 
\begin{figure}[h]
\centering
\includegraphics[width=.5\textwidth]{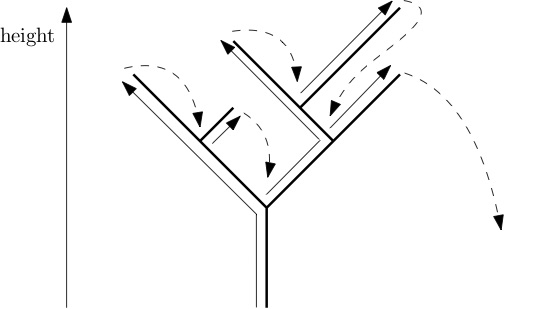}
\end{figure}

Let $E_t := \{v \in \hat\Omega : \tau_v \le t\}$ denote the set of (fully) explored particles and $D_t = \{vi \in \hat\Omega : v \in E_t \} \backslash (E_t \cup \{v_t\})$ denote the discovered particles that are yet to be explored. We then define the length of the tree to be $L = \sup\{\tau_v : v \in \hat \Omega\} = \int_0^\infty N_s ds$.

We extend this notation to an i.i.d. sequence of branching processes, $(X^i)_{i \ge 1}$. For $u = (i, v) \in \mathbb N \times \hat\Omega$, write $b_u$ and $d_u$ for the birth and death times, respectively, of particle $v$ in $X^i$. Similarly, we extend the definitions of $\sigma_u$ and $\tau_u$ to $u \in \mathbb N \times \hat\Omega$, where we understand $(i_1, v_1) \preceq_{\rm lex}(i_2, v_2)$ to mean $i_1v_1 \preceq_{\rm lex} i_2v_2$. Then, there exists a unique $(i, v) = (I_t, v_t)$ such that $t \in [\sigma_{(i, v)}, \tau_{(i, v)}]$, which defines the depth-first exploration process for the sequence of trees. We still write $h_t = b_{(I_t, v_t)} + (t - \sigma_{(I_t, v_t)})$ and $\zeta_t = X_{v_t}^{I_t}(h_t)$ to denote the height process and associated position in $E$, respectively. We also extend the definitions of $E_t$ and $D_t$ in a similar way.

\subsection{Metric measure spaces and convergence}\label{subsec:mmspace}
Here we introduce the necessary notation to allow us to view the genealogical tree associated to a MBP as a (pointed) metric measure space. We mostly follow \cite{felix}; we also refer the reader to \cite{felixmmspace, ellen, mmm-space, gap16}.

In general, a pointed metric measure space is a quadruple $\mathcal X = (X, d, \nu, {\bf r})$, where $(X, d)$ is a complete separable metric space, $\nu$ is a finite measure on $X$ and ${\bf r}$ is the root. We let $\mathbb X$ denote the equivalence classes of (pointed) metric measure space, where $\mathcal X$ and $\mathcal Y$ are equivalent if there exists an isometry 
\[
 \iota : \{{\bf r}_X\} \cup {\rm supp} \, \nu_X \to \{{\bf r}_Y\} \cup {\rm supp} \, \nu_Y
\]
such that $\iota({\bf r}_X)={\bf r}_Y$ and $\nu_X \circ \iota^{-1} = \nu_Y$. 

\medskip

We now introduce the Gromov-weak and Gromov-Hausdorff-weak topologies.

\medskip

\noindent{\bf Gromov-weak topology.} For $k \ge 1$ and $\varphi : [0, \infty)^{k+1 \choose 2} \to \mathbb R$, define $\Phi : \mathbb X \to \mathbb R$ via
\[
 \Phi(\mathcal X) = \int_{X^k} \varphi(d({\bf x}))\mu^{\otimes k}({\rm d}{\bf x}),
\]
where, for ${\bf x} = (x_1, \dots, x_k) \in X^k$, $d({\bf x}) = (d(x_i, x_j))_{0 \le i, j \le k}$ and $x_0 = {\bf r}$. A sequence $(\mathcal X_n)_{n \ge 1}$ is said to converge to $\mathcal X$ in the Gromov-weak topology if, and only if, 
\begin{equation}\label{eq:monomial}
 \Phi(\mathcal X_n) \to \Phi(\mathcal X), \quad n \to \infty,
\end{equation}
for all $\Phi$ defined as in \eqref{eq:monomial} above with $\varphi$ bounded and continuous. 

\medskip

\noindent {\bf Gromov-Hausdorff-weak topology.} We say that $(\mathcal X_n)_{n \ge 1}$ converges to $\mathcal X$ in the Gromov-Hausdorff-weak topology if there exists a complete separable metric space $(Z, d_Z)$ and isometries $\iota_n : \{{\bf r}_n\} \cup {\rm supp}\, \nu_n \to Z$, $n \in \mathbb N \cup \{\infty\}$ such that, as $n \to \infty$
\[
 \iota_n({\bf r}_n) \to \iota_\infty \quad \text{and} \quad \iota_n({\rm supp}\, \nu_n) \to \iota_\infty({\rm supp}\, \nu_\infty),
\]
in $(Z, d_Z)$ and in the Hausdorff topology, respectively, and 
\[
  \nu_n \circ \iota_n^{-1} \to \nu_\infty \circ \iota_\infty^{-1}, \quad n \to \infty,
\]
weakly, as measures on $Z$.

\medskip

In our setting, which will be made more concrete shortly, we will first prove convergence in the Gromov-weak topology and then prove a criterion that will allow us to strengthen the Gromov-weak convergence to Gromov-Hausdorff-weak convergence. Before stating the criterion, let us define the so-called lower mass function, which is defined, for $\delta > 0$ and $\mathcal X \in \mathbb X$ as
\begin{equation}\label{eq:LMF}
  m_\delta(\mathcal X) := \inf\{|B_{\delta, x}(\mathcal X)| : x \in {\rm supp}\, \nu\},
\end{equation}
where $B_{\delta, x}(\mathcal X)$ is the closed ball of radius $\delta$ centred at $x \in X$ {and $|A|:=\nu(A)$ for $A\subset X$.}

\begin{lemma}[Gromov-weak to Gromov-Hausdorff-weak]\label{lem:GW-to-GHW}
    Suppose that $(\mathbb P_n, n \ge 1)$ is a sequence of %locally finite  
    probability measures converging weakly to $\mathbb P$ with respect to the Gromov-weak topology. Suppose that for each $ \delta > 0$, 
    \begin{equation}\label{eq:LMprop}
        \lim_{\eta' \to 0}\limsup_{n \ge 1}\mathbb P_n( m_\delta(\mathcal X) < \eta') = 0.
    \end{equation}
    Then $(\mathbb P_n)_{n \ge 1}$ also converges weakly to $\mathbb{P}$ in the Gromov-Hausdorff-weak topology.
\end{lemma}

For a proof of the above result, we refer the reader to \cite[Lemma 4.5]{felix}. We now formally describe how to build a metric measure space from a MBP.

Given a (critical) MBP, $X = (X_t)_{t \ge 0}$, let $N_\infty$ denote the total progeny of individuals in the Ulam Harris tree of all individuals to have ever been born in $X$, and let $\mathcal U := \{\emptyset := u_0, u_1, \dots, u_{N_{\infty}}\}$ denote the labels of these particles in depth-first order. Further, let $T$ denote the continuous plane tree with vertices given by $\mathcal U$ and branch lengths given by the lifetimes of the particles, i.e. the length of the branch rooted at $u \in \mathcal U$ is $d_u - b_u$. Recall the total tree length, $L$, and the height function, $(h_t)_{t \ge 0}$, defined in the previous section and, for $s, t \in [0, L)$, define
\[
  d(s, t) = h_s + h_t - 2\inf_{u \in [s \wedge t, s \vee t]}h_u.
\]
Then, since $d(s, t) = 0$ if $s$ and $t$ correspond to two times a branch point is visited in the depth-first exploration, $d$ defines a pseudometric on $[0, L)$. Thus, we define the metric space $(T, d)$ to be $([0, L)/\!\!\sim , d)$, where $\sim$ is the equivalence relation defined by $s \sim t \Leftrightarrow d(s, t) = 0$. We equip $(T, d)$ with the measure $\nu$, the push forward of the Lebesgue measure on $[0, L)$ under the equivalence relation $\sim$.

Finally, we define the Brownian continuum random tree. Let ${\bf e} = ({\bf e}_t)_{0 \le t \le \tau}$ denote a Brownian excursion with a certain speed $c=\sigma^2(f)$ (a constant defined below in \eqref{eq:sf}), conditioned to have maximum height at least one, where $\tau$ denotes the length of the excursion (and the height is defined by $\sup_{t\le \tau} \ee_t$). More precisely, the law of $\ee$ is given by the reflected (speed $c$) one-dimensional Brownian motion It\^{o} excursion measure, restricted to paths whose supremum exceeds $1$, and normalised to be a probability measure. Define
\[
  d_{\bf e}(s, t) = {\bf e}_s + {\bf e}_t - 2\min_{r \in [s, t]}{\bf e}_r, \qquad 0 \le s \le t.
\]
Again, this defines a pseudometric on $[0, \tau)$ and thus, $(T_{\bf e}, d_{\bf e}) := ([0, \tau)/\!\!\sim, d_{\bf e})$ defines a metric space, where $\sim$ is the equivalence relation that identifies points with $d_{\bf e}(\cdot, \cdot) = 0$. Similarly to above, we equip this metric space with the measure $\mu_{\bf e}$, the push forward of the Lebesgue measure on $[0, \tau)$.  

\section{Convergence to the CRT}\label{sec:proof}

\subsection{The Martingale}\label{sec:martingale}

In this section we introduce the key martingale $M$, which is an analogue of the Lukasiewicz path for Galton--Watson processes (see for example \cite{LeGall2005}).

We let $\mathbb{P}^{\star}_{\delta_x}$ denote the law of an i.i.d.\,sequence of branching processes started with a single particle at $x$, i.e.\,with law $\mathbb{P}_{\delta_x}$. Recall that $I_t$ denotes the index of the tree being visited at time $t$ in the depth-first exploration and $D_t$ is the set of ``discovered but yet-to-be explored'' particles at time $t$ in the depth-first exploration.

We set 
\[ 
M_t = \varphi(\zeta_t) + \sum_{w\in D_t} \varphi(X_{w}(b_w))-\sum_{w\in E_t \cap \mathbb{N}} \varphi(X_{w}(b_w))
\]
which is a càdlàg process by definition. Note that if $w\in E_t\cap \mathbb{N}$ then $w$ is the root vertex of some tree and $X_{w}(b_w)=x$; i.e.~the third term above can be equivalently written as $(I_t-1)\varphi(x)$.
We also define
\[
m_t = \varphi(\zeta_t)-\varphi(X_{v_t}(b_{v_t})) - \int_0^t (L\varphi)(\zeta_s) \, ds  + \sum_{w\in E_t} \varphi(X_{w}(d_w))-\varphi(X_{w}(b_w)).
\]

One can think of $(M_t)_{t \ge 0}$ as being the martingale $W$ from \eqref{E:W} for the MBP, but explored in depth-first rather than a breadth-first order. We introduce $(m_t)_{t \ge 0}$ for the purpose of proving the (conceptually clear but technically non-trivial) martingale property of $M$, and identifying its quadratic variation explicitly.

\begin{lemma}\label{lem:mgs}
Under Assumption \ref{A:main}, $m_t$ and $M_t$ are local martingales under $\mathbb{P}^{\star}_x$ with respect to the filtration $\mathcal{F}_t $ generated by $(v_s,\zeta_s,D_s,\{X_w(b_w): w\in D_s\})_{s\le t}$. Moreover, their predictable quadratic variations are given by 
	\[
	\langle m \rangle_t = \int_0^t \Gamma\varphi(\zeta_s) ds 	\vspace{-.5cm}
	\]
 and
	\begin{equation}\label{eq:f}
		\langle M \rangle_t = \int_0^t f(\zeta_s) ds \quad \text{ where } \quad f(x) = \Gamma\varphi(x)+\gamma(x)\mathcal{E}_x\bigg[\bigg(\sum_{i=1}^N \varphi(x_i)-\varphi(x)\bigg)^2\bigg].
	\end{equation}
\end{lemma}

\begin{proof}
	We start with the claims concerning $(m_t)_t$. 
	Let $v^{(n)}$ be the $n$-th particle in the lexicographical ordering of the i.i.d. sequence of trees and let $T_n=\tau_{v^{(n)}}$ be the time that depth-first exploration finishes discovering the trajectory of $v^{(n)}$, {setting $T_0:=0$.} We will show by induction that $m_{t\wedge T_n}$ is a martingale for all $n\in \mathbb{N}$. To this end, assume that $m_{t\wedge T_{n-1}}$ is a martingale. {(Note that $m_{t\wedge T_0}\equiv m_0$ is constant and so clearly a martingale).} Then, for $t>s$
\begin{equation}\label{eq:cond_exp}
		 \mathbb{E}_{\delta_x}^\star[m_{t\wedge T_n}|\mathcal{F}_s] = m_{s\wedge T_{n-1}}+\mathbb{E}_{\delta_x}^\star[m_{t\wedge T_n}-m_{t\wedge T_{n-1}}|\mathcal{F}_s].
\end{equation}
	We consider three cases: $s \ge T_n$, $T_{n-1} \le s < T_n$ and $s < T_{n-1}$. In the first case, the right-hand side of \eqref{eq:cond_exp} is equal to $m_{T_{n-1}}+(m_{T_n}-m_{T_{n-1}})=m_{T_n}=m_{s\wedge T_n}$. On the event $s\in [T_{n-1},T_n)$, it follows from \eqref{eq:stoppedCexp} that $\mathbb{E}_{\delta_x}^\star[m_{t\wedge T_n}-m_{t\wedge T_{n-1}}|\mathcal{F}_s]=m_s-m_{T_{n-1}}$ {Indeed, conditionally on $\mathcal{F}_s$, $(\zeta_t)_{t\in [s,T_{n}]}$ evolves as the single particle motion under law $\bfP$, and is stopped either at the first time it leaves $E$ (that is, is sent to $\dagger$) or the first branching event along its trajectory. The latter, conditionally on the motion, occurs at rate $\gamma(\zeta_\cdot)$. Moreover on the interval $[s,T_n]$, $m_t-m_s=\varphi(\zeta_t)-\varphi(\zeta_s)-\int_0^t (L\varphi)(\zeta_s) dx$. In other words the conditional law of $m_{(r+s)\wedge T_n}-m_s$ given $\mathcal{F}_s$ is that of $C_{r\wedge \sigma}$ under $\bfP$ as in \eqref{eq:stoppedCexp} (where we also used that $C$ is constant as soon as the particle reaches $\dagger$). H}ence, the right-hand side of \eqref{eq:cond_exp} is equal to $m_s=m_{s\wedge T_n}$. Finally, if $s<T_{n-1}$, thanks again to \eqref{eq:stoppedCexp}, we have
	\begin{align*}
		\mathbb{E}_{\delta_x}^\star[m_{t\wedge T_n}-m_{t\wedge T_{n-1}}|\mathcal{F}_s]  = \mathbb{E}_{\delta_x}^\star[\mathbb{E}_{\delta_x}^\star [m_{t\wedge T_n}-m_{t\wedge T_{n-1}}|\mathcal{F}_{T_{n-1}}]|\mathcal{F}_s] 
		 = 0
		\end{align*} 
	and so the right-hand side of \eqref{eq:cond_exp} is equal to $m_{s\wedge T_{n-1}}=m_s=m_{s\wedge T_n}$. Thus, $m_{t\wedge T_n}$ is a martingale with respect to $\mathcal{F}_t$ and hence $m_t$ is a local martingale for $\mathcal{F}_t$. The same reasoning together with \eqref{eq:Gamma} yields that 
	\begin{equation}
		\label{eq:mpqv}
		 \langle m \rangle_t = \int_0^t \Gamma (\varphi)(\zeta_s) \, ds
		 \end{equation}
		as desired.

	Now we write
	\begin{align*}
		M_t & = m_t +\int_0^t (L\varphi)(\zeta_s) ds + \sum_{w\in E_t\cup \{v_t\}} \varphi(X_{w}(b_w))-\sum_{w\in E_t} \varphi(X_{w}(d_w)) \\
		& \phantom{= m_t +\int_0^t (L\varphi)(\zeta_s) ds + Xxxxxxxxx} +\sum_{w\in D_t} \varphi(X_{w}(b_w))-\sum_{w\in E_t \cap \mathbb{N}} \varphi(X_{w}(b_w)) \\
		& = m_t +\int_0^t (L\varphi)(\zeta_s) ds  +  \sum_{w\in E_t\cup D_t\cup \{v_t\}\setminus \mathbb{N}} \varphi(X_{w}(b_w))- \sum_{w\in E_t} \varphi(X_{w}(d_w)) \\
	& =	m_t +\int_0^t (L\varphi)(\zeta_s) ds  +  \sum_{w\in E_t} \bigg(\bigg(\sum_{v: v=wi} \varphi(X_{v}(b_v))\bigg)-\varphi(X_{w}(d_w))\bigg) .
		\end{align*}
Since $m_t$ is a martingale, to show that $M_t$ is a martingale, it remains to justify that $M_t-m_t$ is a martingale. Again we will show this by induction. Define $(T_n, n\ge 0)$ as before, and let $t>s>0$. Then 
\begin{multline*}
\mathbb{E}_{\delta_x}^{\star}[M_{t\wedge T_n}-m_{t\wedge T_n}\mid \mathcal{F}_s]=M_{s\wedge T_{n-1}}-m_{s\wedge T_{n-1}} \\+ \mathbb{E}\bigg[\int_{t\wedge T_{n-1}}^{t\wedge T_n} (L\varphi)(\zeta_r) dr + \sum_{w\in E_{t\wedge T_{n}}\setminus E_{t\wedge T_{n-1}}}\bigg(\bigg(\sum_{v:v=wi} \varphi(X_v(b_v))\bigg)-\varphi(X_w(d_w))\bigg) \, \bigg| \, \mathcal{F}_s \bigg]. 
\end{multline*}
If $s\ge T_{n}$ the right hand side is clearly equal to $M_{T_{n}}-m_{T_n}=M_{s\wedge T_n}-m_{s\wedge T_n}$. If $s\in [T_{n-1},T_n)$ then the right hand side is 
\begin{align*}
\phantom{=} &  M_{s}-m_{s}+\mathbb{E}\bigg[\int_{ s}^{t\wedge T_n} (L\varphi)(\zeta_r) dr + \sum_{w\in E_{t\wedge T_{n}}\setminus E_{s}} \bigg(\bigg(\sum_{v:v=wi} \varphi(X_v(b_v))\bigg)-\varphi(X_w(d_w))\bigg) \, \bigg| \, \mathcal{F}_s \bigg] \\
= & M_s-m_s + \mathbb{E}[\tilde{C}_{s,t}\mid \mathcal{F}_s]+\mathbb{E}\bigg[e^{-\int_s^t \gamma(\zeta_r) dr}\varphi(\zeta_t)+\int_s^t \gamma(\zeta_r) e^{-\int_s^r \gamma(\zeta_u) du} m[\varphi](\zeta_r) dr \, \bigg| \, \mathcal{F}_s\bigg]
\end{align*}
where 
\[\tilde{C}_{s,t}=e^{-\int_s^t \gamma(\zeta_r) dr}\bigg(\int_s^t L\varphi(\zeta_r) dr - \varphi(\zeta_t)\bigg)+\int_s^t \gamma(\zeta_r) e^{-\int_s^r \gamma(\zeta_u) du} \bigg(\int_s^r L\varphi(\zeta_u) du - \varphi(\zeta_r)\bigg) dr. \]
Using \eqref{eq:stoppedCexp} together with the fact that $\varphi$ is an eigenfunction for the branching process we conclude that (for $s\in [T_{n-1},T_n)$)
\[
\mathbb{E}_{\delta_x}^{\star}[M_{t\wedge T_n}-m_{t\wedge T_n}\mid \mathcal{F}_s]=M_s-m_s -\varphi(\zeta_s)+\varphi(\zeta_s)=M_{s\wedge T_{n}}-m_{s\wedge T_n}.
\]
The case $s<T_{n-1}$ is argued similarly; we omit the details.

Finally, we prove the claim concerning the quadratic variation of $M_t$. Since the integral term above has finite variation, the predictable quadratic variation of $M_t-m_t$ is the predictable compensator of the process 
	\[ t\mapsto A_t:= \sum_{w\in E_t} \bigg(\bigg(\sum_{v: v=wi} \varphi(X_{v}(b_v))\bigg)-\varphi(X_{w}(d_w))\bigg)^2 .\]
	This is the u.c.p.~limit of 
	\begin{equation}\label{eq: pqv}
	\sum_{t_i<t} \mathbb{E}[A_{t_{i+1}\wedge t}-A_{t_i\wedge t} \mid \mathcal{F}_{t_i}]
	\end{equation}
	as the mesh size of the partition $\{0=t_0<t_1<\dots <t_n=t\}$ converges to $0$.
	For any interval $[u,u+\delta]$ with $t\ge u$, write $U=\inf\{r\ge u:E_r>E_u\}\wedge (u+\delta)$. Then 
	\begin{align*}
		\mathbb{E}[A_{u+\delta \wedge t} - A_u |\mathcal{F}_u]=\mathbb{E}\bigg[\bigg(\sum_{v=ui} \big(\varphi(X_v(b_v))\big)-\varphi(X_u(d_u))\bigg)^2\mathds{1}_{U<u+\delta} \, \bigg| \, \mathcal{F}_u\bigg]+\epsilon(u,\delta),
	\end{align*}
	where 
	\[ \mathbb{E}\bigg[\bigg(\sum_{v=ui} (\varphi(X_v(b_v)))-\varphi(X_u(d_u))\bigg)^2\mathds{1}_{U<u+\delta}\, \bigg| \,  \mathcal{F}_u\bigg] = f(\zeta_u)-\Gamma (\varphi)(\zeta_u)\]
	and 
	\[
	\epsilon(u,\delta) \le \|\varphi\|_\infty  \mathcal{E}[N^2]\,\mathbb{E}[|E_{u+\delta}\setminus E_U] \le C \delta^2,
	\]
	where $C<\infty$ depends only on the law of the branching process. Returning to \eqref{eq: pqv} gives that 
	\[\langle M \rangle_t = \int_0^t f(\zeta_s)-\Gamma (\varphi) (\zeta_s) \, ds,\] 
	which, combined with \eqref{eq:mpqv}, completes the identification of the quadratic variation.
\end{proof}

The following theorem and its corollary will be proved in the next section. 

\begin{thm}\label{thm:MCLT}
Under Assumption \ref{A:main} have
	\[
	\left( \frac{M_{n^2t}}{n}\right)_{t \ge 0} {\overset{d.}{\to}} \left(B_{\sigma^2(f) t}\right)_{t\ge 0}
	\]
	as $n\to \infty$, where 
	$B$ is a standard one-dimensional Brownian motion, and for $f$ defined in \eqref{eq:f}
	\begin{equation}\label{eq:sf}
	\sigma^2(f) := \frac{\langle \tilde\varphi, f \rangle}{\langle \tilde{\varphi}, 1 \rangle}.
    \end{equation}
\end{thm}

\begin{cor}\label{cor:mgalefcltexcursion}
Under Assumption \ref{A:main}, let $(\hat{M}_t)_{t \ge 0}$ be the martingale $(M_t)_{t \ge 0}$ but restricted to the exploration of a single critical tree. Then, for any $c>0$
{
\[ \mathbb{P}_{\delta_x}(\sup \hat{M} >cn)\sim \frac{\varphi(x)}{cn}\]
as $n\to \infty$, and }
\[
\mathcal{L}\left(\frac{\hat{M}_{n^2\cdot}}{n}\, | \, \sup\hat{M}>cn\right) \to \mathcal{L}(\ee\, | \, \sup(\ee)>c), \quad n \to \infty,
\]
where $\ee$ is a Brownian excursion run at speed $\sigma^2(f)$.

\end{cor}

\subsection{Proofs of Theorem \ref{thm:MCLT} and Corollary \ref{cor:mgalefcltexcursion}} \label{sec:fclt}

We assume that Assumption \ref{A:main} holds throughout this section. 
\begin{proof}[Proof of Theorem \ref{thm:MCLT}] To prove this result, we will show that the conditions of Theorem \ref{thm:CLT} are satisfied. Thanks to Lemma \ref{lem:mgs} and the fact that the i.i.d. trees in the exploration process all start at the same position, it is clear that conditions $(a)$ and $(c)$ are satisfied. The following result shows that $(b)$ holds. 

\begin{prop}\label{prop:prove_b}
For any $R < \infty$, 
\[
    \lim_{N\to \infty} \frac{\mathbb{E}[\sup_{t\in [0,N^2R]}|\Delta M_{t}|^2]}{N^2} = 0.
\] 
\end{prop}

\smallskip
	
Finally, the following result implies that condition (d) of Theorem \ref{thm:CLT} holds.

\smallskip

\begin{prop}\label{P:QV}
	Let $f:E\to[0,1]$ be measurable. Then 
	\[
	\frac{1}{t} \int_0^t f(\xi_s) ds \to \sigma^2(f)
	\]
	in probability as $t\to \infty$. 
\end{prop}

This completes the proof of Theorem \ref{thm:MCLT}, modulo proving the two propositions.
\end{proof}

The proof of Proposition \ref{prop:prove_b} is straightforward, so let us give this first. 

\begin{proof}[Proof of Proposition \ref{prop:prove_b}]
	Let \(X_1, X_2, \dots\) be a sequence of i.i.d. random variables, independent of $N_R \sim \mathrm{Poisson}(\lambda R)$\ and satisfying \(\mathbb{E}[X_1^2]<\infty\). For any \(y>0\) and \(n \ge 1\), we have
	\[
	\max_{1\le i \le n} X_i^2 \le y^2 + \sum_{i=1}^n X_i^2 \mathbf{1}_{\{|X_i|>y\}}.
	\]
	Applying this with \(y=\varepsilon\sqrt{R}\) and \(n=N_R\), and taking expectations gives
	\[
	\mathbb{E}\big[\sup_{1\le i \le N_R} X_i^2\big] \le \varepsilon^2 R +\mathbb{E}\!\left[\sum_{i=1}^{N_R} X_i^2 \mathbf{1}_{\{|X_i|>\varepsilon\sqrt{R}\}}\right].
	\]
	Since the \((X_i)_{i \ge 1}\) are i.i.d. and independent of \(N_R\), it follows that
	\[\mathbb{E}[\sup_{1\le i \le N_R} X_i^2]\le \varepsilon^2 R +\mathbb{E}[N_R]\,
	\mathbb{E}\!\left[X_1^2 \mathbf{1}_{\{|X_1|>\varepsilon\sqrt{R}\}}\right]
	= \varepsilon^2 R 
	+ \lambda R\, \mathbb{E}\!\left[X_1^2 \mathbf{1}_{\{|X_1|>\varepsilon\sqrt{R}\}}\right].
	\]
	Divide both sides by \(R\) yields
	\[
	\frac{\mathbb{E}[\sup_{1\le i \le N_R} X_i^2]}{R}\le
	\varepsilon^2 
	+ \lambda\, \mathbb{E}\!\left[X_1^2 \mathbf{1}_{\{|X_1|>\varepsilon\sqrt{R}\}}\right].
	\]
	Since \(\mathbb{E}[X_1^2]<\infty\), dominated convergence implies that
	\(\mathbb{E}[X_1^2 \mathbf{1}_{\{|X_1|>\varepsilon\sqrt{R}\}}]\to0\)
	as \(R\to\infty\).
	Hence
	\[
	\limsup_{R\to\infty}\frac{\mathbb{E}[\sup_{1\le i \le N_R} X_i^2]}{R} \le \varepsilon^2.
	\]
	Letting \(\varepsilon \downarrow 0\) yields
	\[
	\lim_{R\to\infty}\frac{\mathbb{E}[\sup_{1\le i \le N_R} X_i^2]}{R}=0.
	\]
	To apply this to our context, note that the second moments of $\Delta M_t$ are finite thanks to boundedness of $\varphi$ and $\gamma$, and the fact that the offspring distribution has finite variance. The result then follows by noting that the number of jumps of the process $M_t$ for $t \in [0, R]$ can be dominated by a Poisson random variable with rate $\norm{\gamma}_\infty R$. 
\end{proof}

Before proving Proposition \ref{P:QV}, we will need some preliminary lemmas. 

\begin{lemma}\label{lem:length}
	There exists $C\in(0,\infty)$ such that 
	\[ \mathbb{P}_{\delta_x}(L\ge t)\le \frac{C}{\sqrt{t}}\] 
	for all $t\ge 1$ and $x\in E$, where $L$ is defined (in Section \ref{subsec:labels}) as the length of the tree.
\end{lemma}

\begin{proof}
We have
	\begin{align*}
		\mathbb{P}_{\delta_x}(L\ge t) & \le  \mathbb{P}_{\delta_x}(N_{\sqrt{t}}>0)+\mathbb{P}_{\delta_x}(L\ge t\, | \, N_{\sqrt{t}}=0) \\
			& \le \frac{C'}{\sqrt{t}}+\frac{\mathbb{E}_{\delta_x}[\int_0^{\sqrt{t}} \langle 1, X_h\rangle dh]}{t} \\
			& \le \frac{C'}{\sqrt{t}}+\frac{C \sqrt{t}}{t}, 
	\end{align*}	
where we have used Theorems \ref{thm:survival} and \ref{thm:moments} to bound the survival probability and the occupation measure, respectively. The result follows by taking $C=C'+C''$. 
\end{proof}

Recall that $I_t$ denotes the index of the tree being visited at time $t$ in the depth-first exploration and that $\tau_t$ denotes the time that this tree is completely explored (so $\tau_t\ge t$). 
\begin{lemma}\label{C:tailsIt}We have
	\[
	\mathbb{P}_{\delta_x}(I_t \ge M \sqrt{t}) \to 0
	\]
	as $M\to \infty$, uniformly in $t\ge 1$ and $x \in E$. 
\end{lemma}

\begin{proof}
	Write $L_i$ for the total length of the $i$-th tree in the sequence 
\begin{align*}
	\mathbb{P}(I_t \ge M\sqrt{t}) & \le \mathbb{P}\bigg(\sum_{i=1}^{M\sqrt{t}} L_i \ge t\bigg) \\
	& \le \mathbb{P}\bigg(\bigcap_{i=1}^{M\sqrt{t}} \{L_i\ge t\}\bigg) \\
	& \le \bigg(1-\frac{C}{\sqrt{t}}\bigg)^{M\sqrt{t}}.
\end{align*}
Since $(1-C/\sqrt{t})^{\sqrt{t}}\to e^{-C}$ as $t\to \infty$, there exists $q<1$, such that $(1-C/\sqrt{t})^{\sqrt{t}}<q$ for all $t\ge 1$. 
So 
\[ 
\sup_{t\ge 1} \mathbb{P}(I_t\ge M\sqrt{t})\le q^M
\]
which converges to $0$ as $M\to \infty$. 
\end{proof}

\begin{lemma}\label{C:tailstaut}We have
\[
	\mathbb{P}_{\delta_x}(\tau_t \ge M t) \to 0
\]
as $M\to \infty$, uniformly in $t\ge 1$ and $x \in E$. 
\end{lemma}

\begin{proof}
	For any $K>0$, 
\[
\mathbb{P}_{\delta_x}(\tau_t \ge M t) \le \mathbb{P}(I_t\ge K\sqrt{t})+\mathbb{P}(\max\{L_1,\dots, L_{K\sqrt{t}}\}\ge (M-1)t).
\]
We have 
\begin{align*}
\mathbb{P}(\max\{L_1,\dots, L_{K\sqrt{t}}\}\ge (M-1)t) & = 1- \mathbb{P}(\cap_{i=1}^{K\sqrt{t}} \{L_i < (M-1)t\}) \\
& \le 1 - \bigg(1-\frac{C}{\sqrt{(M-1)t}}\bigg)^{K\sqrt{t}}.
\end{align*}
So by Lemma \ref{C:tailsIt}, it is sufficient to prove that for fixed $K>0$, 
\[
\sup_{t\ge 1}\left( 1 - \bigg(1-\frac{C}{\sqrt{(M-1)t}}\bigg)^{K\sqrt{t}}\right) \to 0, \, \text{ as } M\to \infty.
\]
But for this fixed $K$, there exists $q>0$ such that 
\[ \sup_{t\ge 1,M\ge 1} \bigg(1-\frac{C}{\sqrt{(M-1)t}}\bigg)^{K\sqrt{(M-1)t}}\ge q\]
which means that 
\[
\sup_{t\ge 1}\left( 1 - \bigg(1-\frac{C}{\sqrt{(M-1)t}}\bigg)^{K\sqrt{t}}\right) \le 1-q^{1/\sqrt{M-1}}
\]
and this indeed converges to $0$ as $M\to \infty$.
\end{proof}

\begin{lemma}\label{C:tailsJt}
We have 
	\[
\limsup_{t\to \infty} \sup_x \frac{	\mathbb{E}_{\delta_x}\left[\sum_{v\in D_t} \varphi(X_v(b_v))\right]}{\sqrt{t}}<\infty.
	\]
\end{lemma}

\begin{proof}
	Note that from the definition of $(M_t)_{t \ge 0}$, we have 
    \[
      \mathbb E_{\delta_x}\left[\sum_{v \in D_t} \varphi(X_v(b_v))\right] = \mathbb E_{\delta_x}[M_t] + \mathbb E_{\delta_x}\left[\sum_{v \in E_t \cap \mathbb N}\varphi(X_v(b_v))\right] - \mathbb E_{\delta_x}[\varphi(\zeta_t)].
    \]
    Since $(M_t)_{t \ge 0}$ is a martingale we have $\mathbb E[M_t] = \mathbb E[M_0] = \varphi(x)$. This and boundedness of $\varphi$ show that the first and third terms are $o(\sqrt{t})$ as $t \to \infty$, uniformly in $x$. We are thus left to deal with the second term. For this, note that 
    \[
      \mathbb E_{\delta_x}\left[\sum_{v \in E_t \cap \mathbb N}\varphi(X_v(b_v))\right] = \mathbb E_{\delta_x}[I_t - 1]\varphi(x)
    \]
    and so it remains to control $\mathbb E[I_t]$. To this end, fix $M > 0$ and write,
    \[
      \mathbb E[I_t] = \sum_{n \ge 0}\mathbb P(I_t > n) = \sum_{n \le M\sqrt{t}}\mathbb P(I_t > n) + \sum_{n \ge M\sqrt{t}}\mathbb P(I_t > n).
    \]
    The first term on the right-hand side above is trivially bounded above by $M\sqrt{t}$. For the second term, thanks to Lemma \ref{C:tailsIt}, we may choose $M$ sufficiently large so that the second term on the right-hand side above is $o(\sqrt{t})$, uniformly in $t$. Putting the pieces together concludes the proof.
\end{proof}

\begin{proof}[Proof of Proposition \ref{P:QV}]
Write $(X_h^i)_{h\ge 0}$ for the $i$-th branching process (i.e.\,the associated atomic measure) in the i.i.d.\,sequence. Consider all the subprocesses %whose roots have been visited in the depth first exploration before time $t$, but not explored - including the subprocess rooted at the particle corresponding to $\zeta_t$ itself.
rooted at particles in $D_t\cup \{v_t\}$. Enumerate them $j=1,\ldots, J_t=|D_t|+1$, and write $\tilde{X}_h^j$ for the measure associated to the $j$-th (sub)-branching process. Then we have 
\[ \int_0^t f(\zeta_s) ds = \sum_{i=1}^{I_t} \int_0^\infty \frac{\langle  X_h^i , f\rangle}{\langle X_h^i, 1 \rangle}\langle X_h^i, 1 \rangle dh - \sum_{j=1}^{J_t} \int_0^\infty \frac{\langle X_h^i , f\rangle}{\langle X_h^i, 1 \rangle}\langle \tilde{X}_h^j,1 \rangle dh, \]
and 
\[ \sigma^2(f) t = \sum_{i=1}^{I_t} \int_0^\infty \frac{\langle \tilde{\varphi}, f\rangle}{\langle \tilde{\varphi},1 \rangle}\langle  X_h^i , 1 \rangle dh - \sum_{j=1}^{J_t} \int_0^\infty \frac{\langle  \tilde{\varphi}, f \rangle}{\langle \tilde{\varphi}, 1\rangle}\langle  \tilde{X}_h^j,1 \rangle dh. \]

Writing 
\[ 
F_h^i=\frac{\langle  X_h^i , f\rangle}{\langle X_h^i, 1 \rangle}-\frac{\langle \tilde{\varphi}, f\rangle}{\langle \tilde{\varphi},1 \rangle}\, \text{ and } \, \tilde{F}_h^j = \frac{\langle  \tilde{X}_h^j , f\rangle}{\langle \tilde{X}_h^j, 1 \rangle}-\frac{\langle \tilde{\varphi}, f\rangle}{\langle \tilde{\varphi},1 \rangle},
\]
we have 
\[
\frac{1}{t} \left(\int_0^t f(\zeta_s) ds - \sigma^2(f)t\right)=\frac{1}{t} \left(\sum_{i=1}^{I_t} \int_0^\infty F_h^i \langle X_h^i, 1 \rangle dh - \sum_{j=1}^{J_t} \int_0^\infty \tilde{F}_h^j \langle \tilde{X}_h^j, 1 \rangle dh \right).
\]
Now, for any $\delta>0$,
\begin{align*}
\frac{1}{t} \left|\int_0^t f(\zeta_s) ds - \sigma^2(f)t\right|  & \le  \frac{1}{t} \sum_{i=1}^{I_t} \int_0^\infty |F_h^i|\mathrm{1}_{|F_h^i|\le \delta} \langle X_h^i, 1 \rangle dh+ \frac{1}{t} \sum_{i=1}^{I_t}\int_0^\infty |F_h^i|\mathrm{1}_{|F_h^i|>\delta} \langle X_h^i, 1 \rangle dh \\
& + \frac{1}{t} \sum_{j=1}^{J_t} \int_0^\infty |\tilde F_h^j|\mathrm{1}_{|\tilde F_h^j|\le \delta} \langle \tilde X_h^j, 1 \rangle dh+ \frac{1}{t} \sum_{j=1}^{J_t}\int_0^\infty |\tilde F_h^j|\mathrm{1}_{|\tilde F_h^j|>\delta} \langle \tilde X_h^j, 1 \rangle dh
\end{align*}
which, since $|\tilde{F}_h^j|\le c, |F_h^i|\le c$ for some $c$ depending only on $f$, gives 
\[ 
\frac{1}{t} \left|\int_0^t f(\zeta_s) ds - \sigma^2(f)t\right| \le 2\delta \frac{\tau_t}{t} + \frac{1}{t}  \sum_{i=1}^{I_t}\int_0^\infty \mathrm{1}_{|F_h^i|>\delta} \langle X_h^i, 1 \rangle dh  + \frac{1}{t} \sum_{j=1}^{J_t}\int_0^\infty \mathrm{1}_{|\tilde F_h^j|>\delta} \langle \tilde X_h^j, 1 \rangle dh.
\]
By Lemma \ref{C:tailstaut}, it therefore suffices to prove that for any fixed $\delta>0$, 
\begin{equation}\label{E:badF}
\frac{1}{t}  \sum_{i=1}^{I_t}\int_0^\infty \mathrm{1}_{|F_h^i|>\delta} \langle X_h^i, 1 \rangle dh  + \frac{1}{t} \sum_{j=1}^{J_t}\int_0^\infty \mathrm{1}_{|\tilde F_h^j|>\delta} \langle \tilde X_h^j, 1 \rangle dh
=: \frac{E_t(\delta)}{t} \to 0,
\end{equation}
in probability, as $t\to \infty$. For $M\in (0,\infty)$ write 
\[
E_t^M(\delta)=\frac{1}{t}  \sum_{i=1}^{I_t\wedge M\sqrt{t}}\int_0^{M^2\sqrt{t}} \mathrm{1}_{|F_h^i|>\delta} \langle X_h^i, 1 \rangle dh  + \frac{1}{t} \sum_{j=1}^{J_t}\int_0^{M^2\sqrt{t}} \mathrm{1}_{|\tilde F_h^j|>\delta} \langle \tilde X_h^j, 1 \rangle dh
\]
so that for any given $\eta>0$, 
\begin{align*}
\mathbb{P}_{\delta_x}\bigg(\frac{E_t(\delta)}{t} \ge \eta\bigg) & \le \mathbb{P}_{\delta_x}\bigg(\frac{E_t^M(\delta)}{t}\ge \eta\bigg)
 +\mathbb{P}_{\delta_x}(I_t \ge M\sqrt{t})\\
 &\quad+ %\mathbb{P}_{\delta_x}(J_t\ge M\sqrt{t}) 
\mathbb{P}_{\delta_x}(\langle X^i_{M^2\sqrt{t}}, 1 \rangle >0 %\text{ or } \langle \tilde{X}^j_{M\sqrt{t}},1\rangle >0 
\text{ for some } 1\le i\le M\sqrt{t})
\\ &\qquad +\mathbb{P}_{\delta_x}( \langle \tilde{X}^j_{M^2\sqrt{t}},1\rangle >0 
\text{ for some } 1\le j\le J_t).
\end{align*}
The result will thus be proved if we can show that the second, third and fourth terms all tend to $0$ as $M\to \infty$, uniformly in $t$, and that the first term tends to $0$ as $t\to \infty$ for any fixed $\delta, M, \eta > 0$ and $x \in E$. 

For the first term, observe that 
\begin{multline*}
\mathbb{E}_{\delta_x}\bigg(\frac{E^M_t(\delta)}{t}\bigg) \le \frac{M}{\sqrt{t}} \sup_{y\in E} \int_0^{M^2\sqrt{t}} \mathbb{E}_{\delta_y}[\langle X_h, 1 \rangle 1_{|\frac{\langle X_h, f \rangle}{\langle X_h, 1\rangle}-\frac{\langle \tilde \varphi, f \rangle}{\langle \tilde\varphi, 1 \rangle}|\ge \delta } \, | \, \langle X_h, 1 \rangle >0]\mathbb{P}_{\delta_y}(\langle X_h, 1 \rangle >0)  \, dh \\ 
+ \frac{1}{t}\int_0^{M^2\sqrt{t}} \mathbb{E}_{\delta_x}\left[ \sum_{j=1}^{J_t}  \mathbb{E}_{\tilde{X}_0^j} [\langle X_h, 1 \rangle 1_{|\frac{\langle X_h, f \rangle}{\langle X_h, 1\rangle}-\frac{\langle \tilde \varphi, f \rangle}{\langle \tilde\varphi, 1 \rangle}|\ge \delta } \, | \, \langle X_h, 1 \rangle >0]\mathbb{P}_{\tilde{X}_0^j}(\langle X_h, 1 \rangle >0)  \right] dh 
\end{multline*} 

Now by \cite{Yaglom2022}
and Theorem \ref{thm:survival}, we have that for any fixed $\delta>0$
\[ 
\alpha(h):=\sup_{y\in E}  \frac{1}{\varphi(y)} \mathbb{E}_{\delta_y}[\langle X_h, 1 \rangle 1_{|\frac{\langle X_h, f \rangle}{\langle X_h, 1\rangle}-\frac{\langle \tilde \varphi, f \rangle}{\langle \tilde\varphi, 1 \rangle}|\ge \delta } \, | \, \langle X_h, 1 \rangle >0]\mathbb{P}_{\delta_y}(\langle X_h, 1 \rangle >0)  \, dh \to 0,
\] as $h\to \infty$. This gives
\[
\mathbb{E}_{\delta_x}\left[\frac{E^M_t(\delta)}{t}\right] \le \int_0^{M^2\sqrt{t}} \left(\frac{M}{\sqrt{t}} + \frac{\mathbb{E}_{\delta_x}[\sum_{j=1}^{J_t} \varphi(\tilde{X}_0^j)]}{t} \right) \alpha(h) dh, 
\]
so by the Riemann--Lebesgue lemma and Lemma \ref{C:tailsJt}, we see that 
\[
\mathbb{E}_{\delta_x}\left[\frac{E_t^M(\delta)}{t}\right]\to 0,
\]
as $t\to \infty$, for any fixed $M, \delta > 0$ and $x \in E$, and hence $\mathbb{P}_{\delta_x}(E_t^M(\delta)\ge \eta t)\to 0$ as $t\to \infty$ for any fixed $M,\eta, \delta > 0$ and $x \in E$.
 
The second term tends to $0$ as $M\to \infty$ uniformly in $t$ by Lemmas \ref{C:tailsIt} and \ref{C:tailsJt}.  For the third term we use a union bound and Theorem \ref{thm:survival} to deduce that this probability is bounded above by a universal constant times $1/M$, which therefore converges to $0$ as $M\to \infty$. 

For the fourth term, we condition on the positions $\{\zeta_t, \{X_v(b_v); v\in D_t\}\}$ to bound
\[\mathbb{P}_{\delta_x}( \langle \tilde{X}^j_{M^2\sqrt{t}},1\rangle >0 
\text{ for some } 1\le j\le J_t) 
\le \mathbb{E}_{\delta_x}\left[\sum_{v\in D_t} \mathbb{P}_{\tilde{X}_0^j}(N_{M^2\sqrt{t}}>0) + \mathbb{P}_{\zeta_t}(N_{M^2\sqrt{t}}>0)\right],
\]
which converges to $0$ as $M\to \infty$, uniformly in $t$, by Lemma \ref{C:tailsJt}.
\end{proof}

\begin{proof}[Proof of Corollary \ref{cor:mgalefcltexcursion}]
This is a standard corollary of Theorem \ref{thm:CLT}, which can be proven in an identical manner to \cite[Proposition 2.5.2]{LeGall02} or \cite[Proposition 6.13]{ellen}, for instance. We omit the details.
\end{proof}

\subsection{Proof of Theorem \ref{thm:main}}\label{sec:mainproof}
 \noindent{\bf Outline of the proof.} We will prove Theorem \ref{thm:main} via the following steps.
	\begin{itemize}
		\item[Step 1.] 
		We first show that if two particles $v,w$ are chosen uniformly (from the duration of the depth first exploration in a tree conditioned to survive until time $n$) then the difference between the tree distance 
		\[
		h_v+h_w-2\min_{u\preceq v, u\preceq w} h_u
		\]
		and the quantity 
		\[ 
		{\frac{2}{\langle \tilde\varphi, \gamma \mathcal{V}[\varphi]\rangle} \left(\hat M_{t(v)}+ \hat M_{t(w)}-2\min_{s\in [t(w)\wedge t(v), t(w)\vee t(v)]}\hat M_s
        \right)}
		\]
		 tends to $0$ in probability as $n\to \infty$, where $t(u)$ is the time that a particle $u$ is visited in the depth-first exploration. 
		\item[Step 2.] Hence, for any finite $k$, the above holds pairwise for $k$ uniformly chosen particles with arbitrarily high probability as $n\to \infty$. We will then combine this with Theorem \ref{thm:CLT} to prove that the pairwise tree distances of $k$ uniformly chosen particles converge in law to the pairwise distances in the CRT (conditioned to reach height one).
		\item[Step 3.] We will then prove that the lower mass bound \eqref{eq:LMprop} holds. By Lemma \ref{lem:GW-to-GHW}, this yields the result.
\end{itemize}

\subsubsection*{Step 1.}
\label{sec::connection}

Recall $\hat M$ from Corollary \ref{cor:mgalefcltexcursion} defined as the martingale $M$ but restricted to a single tree. We will make the connection between the height function $(h_t)_{t \ge 0}$ and $(\hat M_t^d)_{t \ge 0}$, where $\hat M_t^d = \hat M_t - \varphi(\zeta_t)$ under the law $\mathbb P_{\delta_x}(\cdot | N_n > 0)$. 
To do this, we follow the approach of \cite[Section 6.2]{ellen}.

For a tree $T$ and $u, v \in T$, we say that $u$ is a younger sibling of $v$ if there exists $w \in T$ and $1 \le j < k$ such that $u = wj$ and $v = wk$. Then, for $u \in T$, let $E(u)$ denote the set of younger siblings of the ancestors of $u$. With this notation, we set
\[
 \hat M(u) = \sum_{w \in E(u)}\varphi(X_w(b_w)).
\]
The motivation for this choice of notation comes from the fact that $\hat M(v_t)=\hat{M}_t^d$.
We now introduce the notion of good and bad particles: a good particle will be one for which $\hat M$ has an appropriate law-of-large-numbers behaviour. More precisely, given $\eta>0$, we will say that a pair $(u,n)$ with $u\in \mathcal N_n$ is $\eta$-{\it bad} if 
\begin{equation}\label{eqn::etabad}
	\left|\frac{\hat M(u)}{n}- \tfrac12\langle \tilde\varphi, \gamma\mathcal{V}[\varphi]\rangle\right|>\eta,
\end{equation} 
where $\mathcal V$ was introduced in Section \ref{subsec:results}. A pair $(u, n)$ is then said to be $\eta$-good if it is not $\eta$-bad. We also say, for a given $R\geq 0$ and $n\geq R$, that $(u,n)$ is $\eta_R$-{\it bad} if some $(v,s)$ with $R\leq s \leq n$, $v\in \mathcal N_s$ and $v\prec u$ is $\eta$-bad. That is, $(u,n)$ is $\eta_R$-good if all of the ancestors of $u$ after time $R$ are $\eta$-good. 

We have the following estimate for the proportion of $u\in \mathcal N_n$ such that $(u,n)$ is $\eta_R$-bad.

\begin{prop}\label{propn::etabad} 
Assume Assumption \ref{A:main} holds. Fix $\eps,\eta>0$ and write $\mathcal N_n^{\eta_R}:=
	\{u\in \mathcal N_n: (u,n) \text{ is } \eta_R-\text{bad}\}$, $N_n^{\eta_R} = |\mathcal N_n^{\eta_R}|$. Then we have
	\begin{equation}
		\label{eqn::etafbad}
		\sup_{n \ge R} \mathbb{P}_{\delta_x} \left( \left.\frac{N_n^{\eta_R}}{N_n}>\eps \right| N_n>0 \right) \to 0
	\end{equation}
	as $R\to \infty$, for any $x\in E$.
\end{prop} 

\begin{proof}
	We will first show that for any $\eps>0$, setting 
	\[
	E_{R,n}^\eps:= \bigg\{\frac{\sum_{u\in \mathcal N_n} \vp(X_u(n))\I_{\{(u,n) \text{ is } \eta_R-\text{bad}\}}}{\langle \varphi, X_n\rangle}>\eps\bigg\},
	\] 
	we have
	\begin{equation} \label{eqn::weighted_average_ntbad}
	 \sup_{n\geq R} \mathbb{P}_{\delta_x}\left(E_{R,n}^\eps\, \big| \, N_n>0\right)\to 0,
	 \end{equation} 
	as $R\to \infty$. 
%	The reason for this approach is that it is more convenient (and essentially equivalent) to work with $E_{R,t}^\eps$ than $\{N_t^{\eta_R}/N_t>\eps\}$.

	To show \eqref{eqn::weighted_average_ntbad}, we will use the description of the system under the measure $\mathbb P_{\delta_x}^\varphi$ given in Section \ref{subsec:spine}. Recalling the change of measure \eqref{E:COM}, we see that	
	\[ \mathbb{P}_{\delta_x}(E_{R,n}^\eps \giv N_n>0 ) = \mathbb{P}_{\delta_x}^\varphi \left[\frac{\vp(x)/\mathbb{P}_{\delta_x}(N_n>0)}{\langle \varphi, X_n\rangle} \I_{E_{R,n}^\eps} \right]= \mathbb{P}_{\delta_x}^\varphi \left[ Z_n\,\I_{E_{R,n}^\eps}\right],
	\]
	where $Z_n:=\frac{\vp(x)/\mathbb{P}_{\delta_x}(N_n>0)}{\langle \varphi, X_n\rangle}$. To see that this converges to $0$ it is enough to prove that: 
	\begin{enumerate}[label={}]\setlength{\itemsep}{0em}
\nameditem{(i)}{pf:i} $\sup_{n \ge R}\mathbb{P}_{\delta_x}^\varphi(E_{R,n}^\eps)\to 0$ as $R\to \infty$; and 
\nameditem{(ii)}{pf:ii} for every $\delta>0$, there exists $R'$ and $K$ positive, such that 
		$\mathbb{P}_{\delta_x}^\varphi\left( Z_n \I_{|Z_n|>K} \right) \leq \delta$ for all $n\geq R'$.
\end{enumerate}
		
To show \ref{pf:i} we use the spine decomposition given in Section \ref{subsec:spine} to transfer the probability $\mathbb{P}_{\delta_x}^\varphi(E_{R,n}^\eps)$ to the probability that the spine particle is ``bad'' at sufficiently large times, which is unlikely. For this, let $\tilde{\mathbb{P}}_{\delta_x}^\varphi$ denote the law of the spine position and spine label $(Y_r,\Upsilon_r)_{r\ge 0}$, as constructed in Proposition \ref{prop:spine}. That is,  when restricted to the filtration $\mathcal F_\infty$ containing just the information about the branching process, $\tilde{\mathbb{P}}^\varphi$ is simply $\mathbb P^\varphi$. From Proposition \ref{prop:spine}, namely that the spine is chosen with probability proportional to $\varphi$ at each branching event, we have that 
\[
  \frac{\sum_{u \in \mathcal N_n}\varphi(X_u(n))\mathbf 1_{\{(u, n) \text{ is } \eta_R-\text{bad}\}}}{\langle \varphi, X_n\rangle} = \tilde{\mathbb E}_{\delta_x}[(\Upsilon_n, n) \text{ is } \eta_R-\text{bad} \mid  \mathcal F_n]. 
\]
This and Markov's inequality yield
\begin{align*}
\mathbb{P}_{\delta_x}^\varphi(E_{R,n}^\eps) &= \tilde{\mathbb{P}}_{\delta_x}^\varphi(E_{R,n}^\eps) \\
&\le \frac1\varepsilon \tilde{\mathbb{E}}_{\delta_x}^\varphi\left[ \frac{\sum_{u \in \mathcal N_n}\varphi(X_u(n))\mathbf 1_{\{(u, n) \text{ is } \eta_R-\text{bad}\}}}{\langle \varphi, X_n\rangle} \right] \\
&= \frac1\varepsilon\tilde{\mathbb P}_{\delta_x}^\varphi[(\Upsilon_n, n) \text{ is } \eta_R-\text{bad}]
\end{align*}
and so to show \ref{pf:i}, it suffices to prove that 
\begin{equation} \label{eqn::spine_not_bad} \sup_{n\geq R}\tilde{\mathbb{P}}_{\delta_x}^\varphi((\Upsilon_n,n) \text{ is } \eta_R-\text{bad})\to 0\end{equation} 
as $R\to \infty$. {Recalling the definition of being $\eta_R$-bad, \eqref{eqn::spine_not_bad} is equivalent to showing that 
\begin{equation}\label{eqn::etabadequiv}
\frac{\hat M(\Upsilon_s)}{s} \to \frac12\langle \tilde\varphi, \gamma\mathcal{V}[\varphi]\rangle
\quad \tilde{\mathbb{P}}_{\delta_x}^\varphi	-\text{a.s. as } s\to \infty.
\end{equation}
For this, we use the explicit description of the spine decomposition. This entails that, conditional on the spine motion $(Y_s)_{s \ge 0}$, births along the spine occur at rate $\rho(Y_s)=\gamma(Y_s)m[\varphi](Y_s)/\varphi(Y_s)$, and the positions of created children, described by a point measure $\mathcal{Z}$, has law $\mathcal{P}_{Y_s}^\varphi$ with 
		\[
		\frac{\mathrm{d}\mathcal{P}_x^\varphi}{\mathrm{d} \mathcal{P}_x} = \frac{\langle \varphi, \mathcal{Z} \rangle}{m[\varphi](x)}.
		\]
        Since conditionally on $\mathcal{Z}$ the spine particle is particle $i$ with probability $\varphi(x_i)/\langle \varphi,\mathcal{Z}\rangle$, the positions of created children who are \emph{younger siblings of the spine} is described by a point measure $\mathcal{Z}'$ with law $\nu_{Y_s}$ satisfying 
        \[
        \nu_x(f):=\nu_x(\langle f, \mathcal{Z}'\rangle ):=\frac{1}{m[\varphi](x)}\mathcal{E}_x(\sum_{i=1}^N \varphi(x_i)\sum_{j<i} f(x_j)). 
        \]
        Putting this together we conclude that, conditional on the motion of the spine, 
        groups of ``younger siblings'' (viewed as point measures $\mu$) appear  according to an inhomogeneous Poisson process with intensity 
\[
\frac{\gamma(Y_s)m[\varphi](Y_s)}{\varphi(Y_s)} ds \times \nu_{Y_s}(d\mu).
\]
In particular, conditionally on $(Y_s)_{s\ge 0}$ 
\[\hat{M}(\Upsilon_s)-\int_0^s \frac{\gamma(Y_u)m[\varphi](Y_u)}{\varphi(Y_u)}\nu_{Y_u}(\varphi) du=\hat{M}(\Upsilon_s)-\frac{1}{2}\int_0^s \frac{\gamma(Y_u)}{\varphi(Y_u)}\mathcal{V}[\varphi](Y_u) du
\]
is a martingale with variance 
\[\int_0^s \frac{\gamma(Y_u)m[\varphi](Y_u)}{\varphi(Y_u)}\nu_{Y_u}(\varphi^2) du. \]
By Doob's $L^2$ inequality we therefore see that 
\[
\tilde{\mathbb{P}}_{\delta_x}^\varphi\left(  \sup_{[s_1,s_2]} \left|\frac{\hat{M}(\Upsilon_s)}{s}-\frac{\frac{1}{2}\int_0^s \frac{\gamma(Y_u)}{\varphi(Y_u)}\mathcal{V}[\varphi](Y_u) du}{s}\right|>\eta \, \Bigg| \,  (Y_r)_{r\ge 0} \right)\le \frac{\int_0^{s_2} \frac{\gamma(Y_u)m[\varphi](Y_u)}{\varphi(Y_u)}\nu_{Y_u}(\varphi^2) du}{s_1^2\eta^2}
\]
and taking expectations over the spine motion, using Lemma \ref{lem:ergodic}, we obtain that 
\[
\tilde{\mathbb{P}}_{\delta_x}^\varphi\left(  \sup_{[s_1,s_2]} \left|\frac{\hat{M}(\Upsilon_s)}{s}-\frac{\frac{1}{2}\int_0^s \frac{\gamma(Y_u)}{\varphi(Y_u)}\mathcal{V}[\varphi](Y_u) du}{s}\right|>\eta \right)\le C\frac{s_2}{s_1^2\eta^2}
\]
for some universal constant $C<\infty$. Finally, applying Borel--Cantelli with $s_1=2^N$ and $s_2=2^{N+1}$ for $N\in \mathbb{N}$ we can deduce that 
\begin{equation}\label{eq:spinelln}
\left|\frac{\hat{M}(\Upsilon_s)-\frac{1}{2}\int_0^s \frac{\gamma(Y_u)}{\varphi(Y_u)}\mathcal{V}[\varphi](Y_u) du}{s}\right|\to 0
 \quad  \tilde{\mathbb{P}}^\varphi_{\delta_x}\text{-a.s.} \text{ as } s\to \infty.
 \end{equation}
On the other hand, due to ergodicity of the spine (in particular, Lemma \ref{lem:ergodic}) it follows that
\begin{equation}\label{eq:spinelln2}
 \frac1s\int_0^s \frac{\gamma(Y_u)}{\varphi(Y_u)}\frac12 \mathcal V[\varphi](Y_u)du \to \langle \tilde\varphi , \tfrac{\gamma}{2}\mathcal V[\varphi]\rangle \quad \tilde{\mathbf{P}}_{x}^\varphi\text{-a.s.}
\end{equation}
as $s\to \infty$. Combining with \eqref{eq:spinelln} we obtain \eqref{eqn::etabadequiv} and therefore \ref{pf:i}.}

To prove \ref{pf:ii}, we again appeal to the change of measure \eqref{E:COM} to obtain
\[
 \mathbb P_{\delta_x}^\varphi[\mathbf {1}_{\{|Z_n| > K\}}] = \frac{\mathbb P_{\delta_x}(|Z_n| > k ; N_n > 0)}{\mathbb P_{\delta_x}(N_n > 0)} = \mathbb P_{\delta_x}(|Z_n| > K | N_n > 0).
\]
Since $\varphi(x)/(n\mathbb{P}_{\delta_x}(N_n>0))$ is uniformly bounded above for $n \ge 1$, say, thanks to Theorem \ref{thm:survival}, it remains to show that for any $\delta>0$ there exists $K$ and $R'$ such that  
	\begin{equation}\label{eqn::unifexample} \sup_{n\geq R'}\,\mathbb{P}_{\delta_x}\left( \frac{\langle \varphi, X_n\rangle}{n}< \frac1K\giv|N_n|>0 \right) \leq \delta, \end{equation}
however, this follows from Theorem \ref{thm:Yaglom}. This completes the proof of \eqref{eqn::weighted_average_ntbad}.

We will now deduce from \eqref{eqn::weighted_average_ntbad} that 
\[ \sup_{n \ge R} \mathbb{P}_{\delta_x} \left( \left.\frac{N_n^{\eta_R}}{N_n}>\eps \right| N_n>0 \right) \to 0 \text{ as } R\to \infty.\]

For $r > 0$ define $D_r := \{y \in E : \varphi(y) < 1/r\}$, $\mathcal N_n^{D_r} := \{u \in \mathcal N_n : X_u(t) \in D_r\}$ and $N_n^{D_r} := |\mathcal N_n^{D_r}|$. Then, we have $N_n^{\eta_R}\leq N_n^{D_r}+ |\mathcal N_n^{\eta_R} \setminus (\mathcal N_n^{\eta_n}\cap \mathcal N_n^{D_r})|$. 
Using the same ideas as for the proof of Corollary \ref{cor:joint_Yaglom} (see also \cite[Corollary 1.7]{ellen}), for $\delta > 0$ there exists $r > 0$ such that
\[
  \sup_{n \ge R}\mathbb P_{\delta_x}\left( \frac{N_n^{D_r}}{N_n} > \frac{\varepsilon}{2}\giv N_n > 0\right) \le \delta/2,
\]
for $R$ sufficiently large. Then, note that 
	$$\frac{|\mathcal N_n^{\eta_R}\setminus (\mathcal N_n^{\eta_n}\cap \mathcal N_n^{D_r})|}{N_n}\leq \frac{\sum_{u\in \mathcal N_n} \vp(X_u(t))\I_{\{(u,t) \text{ is } \eta_R-\text{bad}\}}}{\sum_{u\in \mathcal N_n} \vp(X_u(t))} \|\vp\|_\infty r$$
so that, by (\ref{eqn::weighted_average_ntbad}),
	$$ \sup_{n\ge R}\mathbb{P}_{\delta_x}\bigg( \frac{|\mathcal N_n^{\eta_R}\setminus (\mathcal N_n^{\eta_n}\cap \mathcal N_n^{D_r})|}{N_n} > \eps/2 \bigg) < \delta/2$$
	for all $R$ sufficiently large. Putting the pieces together concludes the proof.
\end{proof}
\bigskip 

The next two lemmas build on Proposition \ref{propn::etabad} to provide the key connection between $(\hat M_t^d)_{t\ge 0}=(\hat{M}_t-\varphi(\zeta_t))_{t\ge 0}$ and the height function $(h_t)_{t\ge 0}$ for a critical branching diffusion under $\mathbb{P}_{\delta_x}(\cdot | N_n > 0)$. {The first lemma says that conditioning on survival until time $n$ (equivalently, on the supremum of the height function exceeding $n$) is equivalent to conditioning on the maximum of $\hat{M}$ (or $\hat{M}^d$) exceeding $\tfrac12\langle \tilde\varphi, \gamma \mathcal{V}[\varphi]\rangle n$, and is a simple consequence of what we have already proved.}

{\begin{lemma}\label{lem::equiv_cond} 
Assume Assumption \ref{A:main} holds. Then we have
	\begin{equation}
		\label{eqn::etafbad}
		\mathbb{P}_{\delta_x} \left( \sup_t \hat{M}^d_t > \frac{\langle \tilde\varphi,\gamma\mathcal{V}[\varphi] \rangle}{2} n \, \Big|\,  N_n>0\right) \to 1 \text{ and } \mathbb{P}_{\delta_x} \left( N_n>0 \, \Big| \, \sup_t \hat{M}^d_t > \frac{\langle \tilde\varphi,\gamma\mathcal{V}[\varphi] \rangle}{2} n \right)\to 1
	\end{equation}
	as $n\to \infty$, for any $x\in E$.
\end{lemma} 
\begin{proof} The first convergence follows from Proposition \ref{propn::etabad} and the fact that \[\mathbb{P}_{\delta_x}(N_n>0)/\mathbb{P}_{\delta_x}(N_{n(1+\eps)}>0)\to 1\] as $n\to \infty$ for any $\eps$. The second then follows from Theorem \ref{thm:survival} and Corollary \ref{cor:mgalefcltexcursion}, which imply that 
\[ \mathbb{P}_{\delta_x}\left(\sup_t \hat{M}^d_t > \frac{\langle \tilde\varphi,\gamma\mathcal{V}[\varphi] \rangle}{2} n\right) \sim \mathbb{P}_{\delta_x}(N_n>0) \text{ as } n\to \infty.\]

\end{proof}}

{For the second lemma, we let} $\mathbb P_{\delta_x}^1$ denote the law of a tree $T$ (see Section \ref{subsec:mmspace}) under $\mathbb{P}_{\delta_x}$ plus a random variable $t^1$ which, conditionally on $T$, is chosen uniformly from $[0,L(T)):=[0,L)$.

\begin{lemma} \label{lemma::goodcondtree}
Suppose Assumption \ref{A:main} holds. For any $\eta>0$ we have 
	\begin{equation*} 
		\lim_{R\to \infty}\lim_{n \to \infty} \mathbb{P}^1_{\delta_x}\big( (v_{t^1},h_{t^1}) \text{ is } \eta_R-\text{bad} \giv N_n>0 \big) = 0.
	\end{equation*}
\end{lemma}

\begin{rem}
In the case that $h_{t^1}< R$ we also say that $(v_{t^1},h_{t^1})$ is $\eta_R$-bad (with an abuse of notation; recall we only defined the notion of $\eta_R$-badness for $(u,n)$ with $n\geq R$). However, since the probability of this event goes to $0$ as $n\to \infty$ for any fixed $R$, this will not play a role; we only introduce the convention for notational convenience. 
\end{rem}

\begin{proof}
{This essentially follows from Proposition \ref{propn::etabad}, but we first need to do some work to rule out bad situations where we could not appeal to it.}   More precisely, we decompose the probability space into the three disjoint events $\{L\le an^2\}$, $\{L>an^2\}\cap \{h_{t^1}\notin [cn,Cn]\}$ and $\{L>an^2\}\cap \{h_{t^1}\in [cn,Cn]\}$. Then we can write, for any choice of $a,c,C>0$, 
	\begin{align}\label{eqn::uparticleRbad}	\mathbb{P}^1_{\delta_x}\big( & (v_{t^1},h_{t^1}) \text{ is } \eta_R-\text{bad} \giv N_n>0 \big) \notag \\
    &\leq \mathbb{P}^1_{\delta_x}\big( L\leq a n^2 \giv N_n>0\big) + \mathbb{P}^1_{\delta_x}\big( \{h_{t^1} \notin [cn,Cn]\}\cap\{L>an^2\}\giv N_n>0\big) \nonumber \\
		& \quad + \mathbb{P}^1_{\delta_x}\big( \{(v_{t^1},H_{t^1}) \text{ is } \eta_R-\text{bad}\}\cap \{L>an^2\} \cap \{h_{t^1}\in[cn,Cn]\}\giv N_n >0\big).
	\end{align} 
    We will deal with each term separately. {The first two terms correspond to the bad situations mentioned above, and we will show that their probabilities can be made arbitrarily small, uniformly in $n$, by taking first $a\to 0$ and then $c\to 0, C\to \infty$}.

   {To see that $\mathbb{P}_{\delta_x}(L\le an^2 | N_n>0)\to 0$ as $a\to 0$, uniformly in $n$, we first note  the probability is zero for $n < 1/a$ (which is useful as it effectively allows us to restrict to large $n$). Now fix $\eps>0$. If $\tau$ denotes the length of a Brownian excursion {\bf e} with variance $\sigma^2(f)$ and conditioned to reach height $\tfrac{1}{2}\langle \tilde\varphi,\gamma\mathcal{V}[\varphi]\rangle n$, then we can choose $\delta > 0$ such that $\mathbb P(\tau \le \delta) < \varepsilon / 2$. Moreover, thanks to  Corollary \ref{cor:mgalefcltexcursion} and Lemma \ref{lem::equiv_cond}, there exists $N\ge 1$ such that for all $n \ge N$ and $x \in E$, 
    \[
    \mathbb P_{\delta_x}(L/n^2 \le \delta | N_n > 0) - \mathbb P(\tau \le \delta) \le \varepsilon / 2 .
    \]
 Hence for  $a\le  \min\{\delta, 1/N\}$, we have $
     \mathbb P_{\delta_x}(L \le an^2 | N_n > 0) \le \varepsilon$ for all $n$.}

    {We next show that for fixed $a>0$, the second term on the right of \eqref{eqn::uparticleRbad} converges to $0$ as $c\to 0,C\to \infty$, uniformly in $n$. For this, we first notice that $\mathbb{P}_{\delta_x}^1(h_{t^1}>Cn \, | \, N_n>0) \le \mathbb{P}_{\delta_x}(N_{Cn}>0 \, |\, N_n>0) \to 0$ as $C\to \infty$, uniformly in $n$, thanks to Theorem \ref{thm:survival}.} {So we may focus soley on $\mathbb{P}^1_{\delta_x}(\{h_{t^1}\leq cn\}\cap\{L>an^2\} \giv N_n>0)$}. By definition of the conditional probability we have 
	\begin{equation}\label{eq:term2bound1}
		\mathbb{P}^1_{\delta_x}\big(\{h_{t^1}\leq cn\}\cap\{L>an^2\} \giv N_n>0 \big) \le (\mathbb P_{\delta_x}(N_n>0))^{-1} \, \mathbb P_{\delta_x}(\{L\geq an^2\}\cap\{h_{t^1}\leq cn\}),
	\end{equation} 
    which, by conditioning on $\mathcal{F}_\infty$, is equal to 
	\begin{equation*} 
		\mathbb E_{\delta_x}(\I_{\{L\geq an^2\}} \mathbb{P}^1_{\delta_x}(h_{t^1}\leq cn \giv \mathcal{F}_\infty)). 
	\end{equation*}
	Using the lower bound on $L$ from the indicator function, the lower bound on $\mathbb{P}_{\delta_x}(N_n>0)$ from Theorem \ref{thm:survival}, and the fact that $\sup_{x\in E} \sup_{u>0} \mathbb E_{\delta_x}[N_u]<\infty$ (thanks to \ref{a:PF}), it follows that the right-hand side of \eqref{eq:term2bound1} is bounded above by
	\begin{equation*} (\mathbb P_{\delta_x}(N_n>0))^{-1} \,  \mathbb E_{\delta_x}\bigg[\I_{\{L\geq an^2\}} \frac{\int_0^{cn} N_s \, ds}{L}\bigg]  \leq \frac{cK}{a\vp(x)}
	\end{equation*}
	for $K < \infty$ and $n\geq 1$. Hence, for any $a>0$, the second term on the right-hand side of \eqref{eqn::uparticleRbad} converges to $0$ uniformly in $n$ as $c\to 0,C\to \infty.$

    Given the above, we are left to prove that for any fixed $a,c,C>0$,
	\begin{equation}\label{eqn::u_vertex_bad_finaleq}
		\lim_{R\to \infty}\lim_{n\to \infty}\mathbb{P}^1_{\delta_x}\big( \{(v_{t^1},h_{t^1}) \text{ is } \eta_R-\text{bad}\}\cap \{L>an^2\} \cap \{h_{t^1}\in[cn,Cn]\}\giv N_n >0\big)=0.
	\end{equation}

    From now on we assume that $cn\geq R$ (which is without loss of generality, since we are letting $n\to \infty$ before $R$). By conditioning on $\mathcal{F}_\infty$ again, we see that 
	\begin{align}\label{eqn::nrbadeq} 
     \mathbb{P}^1_{\delta_x}&\big( \{(v_{t^1},h_{t^1}) \text{ is } \eta_R-\text{bad}\}\cap \{L>an^2\} \cap \{h_{t^1}\in[cn,Cn]\}\giv N_n>0\big) \nonumber \\ & = \mathbb E_{\delta_x}\left[\I_{\{L> an^2\}}\int_{cn}^{Cn} \frac{N_u^{\eta_R}}{N_u}\frac{N_u}{L}\, du \, \giv N_n>0\right] \nonumber \\
		& \leq \delta + (an^2)^{-1}\mathbb E_{\delta_x}\left[\int_{cn}^{Cn} \I_{\{\frac{N_u^{\eta_R}}{N_u}>\delta\}}N_u \, du \giv N_n>0\right]
	\end{align}
	for any $\delta>0$. Therefore, we just need to prove that the second term in the last expression of \eqref{eqn::nrbadeq}, for fixed $\delta$, converges to 0 as $n\to \infty$ and then $R\to \infty$. However, we can write this term (by Fubini) as 
	\begin{equation}\label{eq:term3bound1}
		(an^2)^{-1} \int_{cn}^{Cn} \mathbb E_{\delta_x}\left[N_u\I_{\{\frac{N_u^{\eta_R}}{N_u}>\delta\}} \giv N_n>0\right] \, du 
    \end{equation}
Now note that for any $n$ such that $cn\wedge n\geq R$ 
	\begin{align*} \mathbb{P}_{\delta_x} \bigg(\frac{N_{cn}^{\eta_R}}{N_{cn}}>\eps \,; \, N_{cn}>0 \giv N_n>0 \bigg) & =  \frac{\mathbb{P}_{\delta_x} \bigg( \frac{N_{cn}^{\eta_R}}{N_{cn}}>\eps\, ; \, N_{cn}>0\, ;  \, N_n >0 \bigg)}{\mathbb{P}_{\delta_x}(N_{n}>0)} \\
		& \leq  \mathbb{P}_{\delta_x} \bigg(\frac{N_{cn}^{\eta_R}}{N_{cn}}>\eps\, ; \, N_{n}>0 \giv N_{cn}>0 \bigg) \frac{\mathbb P_{\delta_x}(N_{cn} >0)}{\mathbb P_{\delta_x}(N_n>0)} \\
		& \leq \sup_{s\geq R} \mathbb{P}_{\delta_x} \bigg( \frac{N_s^{\eta_R}}{N_s}>\eps \giv N_s>0 \bigg)\frac{\mathbb P_{\delta_x}(N_{cn}>0)}{\mathbb P_{\delta_x}(N_n>0)} \\
		& \leq \frac{F(R)}{c} \, \sup_{s\geq R} \mathbb{P}_{\delta_x} \bigg( \frac{N_s^{\eta_R}}{N_s}>\eps \giv N_s>0 \bigg),
	\end{align*}
    where
	$$F(R)=\frac{\sup_{s\geq R} s\mathbb P_{\delta_x}(N_s>0) }{\inf_{s\ge R} s\mathbb P_{\delta_x}(N_s>0)}\longrightarrow 1,$$ 
    as $R\to \infty$, by Theorem \ref{thm:survival}.
Thanks to this, Cauchy-Schwarz and the fact that $cn \ge R$, \eqref{eq:term3bound1} is bounded above by 
\begin{equation*}
\frac{\sqrt{F(R)}}{\sqrt{c}}\sup_{s\geq R} \mathbb P_{\delta_x}\bigg(\frac{N_s^{\eta_R}}{N_s}>\delta \giv N_s>0\bigg)^{1/2} \times (an^2)^{-1}\int_{cn}^{Cn} \mathbb P_{\delta_x}(N_u^2|N_n>0)^{1/2} \, du.
\end{equation*}  
Finally, we observe that by Theorems \ref{thm:moments} and \ref{thm:survival}, we have $\mathbb E_{\delta_x}[N_u^2|N_n>0]^{1/2}\leq Mu$ for some constant $M=M(c,C')$ and for any $u\in [cn,Cn]$, $cn\geq R\geq 1$. Hence, by integrating and applying Proposition \ref{propn::etabad}, we obtain \eqref{eqn::u_vertex_bad_finaleq}.
\end{proof}

{Lemma \ref{lemma::goodcondtree} means that for some $p(R)\to 1$ as $R\to \infty$,  if we pick a particle uniformly at random from a critical MBP conditioned to reach height at least $n$, then all along its ancestry after an initial height $R$, $\hat{M}$ will behave like a constant multiple of height (i.e.~time/generation in the original MBP). As we will explain in the next step, this means that if we pick two (or $k\ge 2$) particles at random, since their genealogical/tree distance will be a function of their two ancestries after time $R$ with high probability as $n\to \infty$, we can approximate their genealogical/tree distance by a `distance' measured using $\hat{M}$ (see \eqref{eqn::matrices}). }

\subsubsection*{Step 2.}\label{subsec:pairwise}
Let $\mathbb{P}_{\delta_x}^k$ be the law of a tree $T$ under $\mathbb{P}_{\delta_x}$, together with $k$ random variables $(t^1,\cdots, t^k)$ chosen conditionally independently and uniformly at random from $[0,L(T))=[0,L)$. We also define the $k\times k$ matrices 
\begin{align}\label{eqn::matrices} 
    (D_t^M)_{ij} & := t^{-1}\big(\hat M^d(v_{t^i})+\hat M^d(v_{t^j})-2\hat M^d(v^{ij}) \big),\\ \nonumber
	(D_t^h)_{ij} & := t^{-1} \big(h_{t^i}+h_{t^j}-2h^{ij}\big),
\end{align}
where $v^{ij}=v_{t^i}\wedge v_{t^j}$ is the most recent common ancestor of $v_{t^i}$ and $v_{t^j}$, and $h^{ij}=d_{v^{ij}}$ is its ``death time". The next proposition says that conditioned on survival up to a large time $t$, these matrices are essentially the same up to a constant.

\begin{prop}\label{prop::distancematrices}
	Suppose Assumption \ref{A:main} holds. Let $k\geq 1$ and $D_t^M$ and $D_t^h$ be as defined above. Then for any $\eps>0$ 
	\begin{equation}\label{eqn::distsecondconv}
		\mathbb{P}_{\delta_x}^k\left(\left\|\tfrac12{\langle \gamma\mathcal V[\varphi], \tilde\varphi\rangle}D_n^h-D_n^M\right\|>\eps \giv N_n>0 \right) \to 0
	\end{equation}
	as $n\to \infty$, where the distance is the Euclidean distance between $k\times k$ matrices.
\end{prop} 

\begin{proof} 
	We prove this in the case $k=2$ with the general result following by a union bound. Note that by symmetry, and since $(D_n^h)_{ii}=(D_n^M)_{ii}=0$ for $i=1,2$, we need only control the distance $\left|\tfrac12{\langle \gamma\mathcal V[\varphi], \tilde\varphi\rangle}(D_n^h)_{12}-(D_n^M)_{12}\right|$. Given $\delta>0$, we:
	\begin{itemize}
		\item choose $M\geq 1$ such that $\lim_{n \to \infty} \mathbb P_{\delta_x}(N_{Mn}>0\giv N_n>0)\leq \delta/4$, which is possible by Theorem \ref{thm:survival};
		\item set $\eta:=\eps/(4M)$ and pick $R$ large enough that $\lim_{n \to \infty}\mathbb{P}_{\delta_x}^2( \{(v_{t^1},h_{t^1}) \text{ is } \eta_R-\text{bad}\}\cup \{(v_{t^2},h_{t^2}) \text{ is } \eta_R-\text{bad}\} \giv N_n>0)\leq \delta/4$, which is possible by a union bound and Lemma \ref{lemma::goodcondtree};
		\item pick $K$ large enough that $\lim_{n \to \infty} \mathbb P_{\delta_x}\big(\int_{0}^R N_s ds\geq K \giv N_n>0\big) \leq \delta/4$, which is possible by Proposition \ref{prop:Q-process}, Markov's inequality and Theorem \ref{thm:moments}.
	\end{itemize} 
	
	Putting this together, we see that for all $t$ large enough, setting
	\begin{equation*} 
    A_n:= \{N_{Mn}>0\} \cup \{\int\nolimits_{[0,R]} N_s ds\geq K\} \cup\{(v_{t^1},h_{t^1}) \text{ is } \eta_R-\text{bad}\}\cup \{(v_{t^2},h_{t^2}) \text{ is } \eta_R-\text{bad}\} \end{equation*} then $	\mathbb{P}_{\delta_x}^2(A_n\giv N_n>0)  \leq \delta $. This implies that $\mathbb{P}_{\delta_x}^2(A_n\giv N_n>0) \to 0$ as $n\to \infty$.
	
	Now consider the complementary event $A_n^c$. From the relevant definitions, we know that on this event we have
	\begin{equation}\label{eqn::bounds_1_matrices}
		h_{t^1}\leq Mn\, , \;\;
		h_{t^2}\leq Mn \, , \;\; 
		\left|\frac{{\hat M^d(v_{t^1})}}{h_{t^1}}-\frac{\langle \gamma\mathcal V[\varphi], \tilde\varphi\rangle}{2}\right|\leq \eta\, , \; \text{ and }\; \left|\frac{{\hat M^d(v_{t^2})}}{h_{t^2}}-\frac{\langle \gamma\mathcal V[\varphi], \tilde\varphi\rangle}{2}\right|\leq \eta,
	\end{equation}
	where $\eta=\eps/(4M)$. Then we consider the two possibilities $B:=\{h^{ij}>R\}$ or $B^c:=\{h^{ij}\leq R\}$. We will show that for large enough $n$, $B\cap A_n^c=B^c\cap A_n^c=\emptyset$, thus completing the proof.
	
	To do this, observe that on $B\cap A_n^c$, by definition of $A_n$ and ``$\eta_R$-badness", we have that
	\begin{equation*}
		h^{12}\leq Mn \; \text{ and } \left|\frac{{\hat M^d(v^{12})}}{h^{12}}-\frac{\langle \gamma\mathcal V[\varphi], \tilde\varphi\rangle}{2}\right|\leq \eta.
	\end{equation*}
	Putting this together with (\ref{eqn::bounds_1_matrices}) gives the deterministic bound
	$$\left|(D_n^M)_{12}-\frac{\langle \gamma\mathcal  V[\varphi], \tilde\varphi\rangle}{2}(D_n^h)_{12}\right| \leq \frac{1}{n} \times 4Mn \times \eta \leq \eps \quad \text{on} \quad B\cap A_n^c$$ 
    and so we have $B\cap A_n^c = \emptyset$.  Furthermore, on the event $B^c\cap A_n^c$, we have $h^{ij}\leq R$ and ${\hat M^d(v^{ij})}\leq \|\vp\|_\infty K$ (by a very crude bound). This means that 
	$$ \left|{\hat M^d(v^{12})}-\frac{\langle \gamma\mathcal V[\varphi], \tilde\varphi\rangle}{2}h^{12}\right| \leq K + Rb^{-1} $$ 
    and so 
	$$ \left|(D_n^M)_{12}-\frac{\langle \gamma\mathcal V[\varphi], \tilde\varphi\rangle}{2}(D_n^h)_{12}\right| \leq \frac{\eps}{2}+ \frac{K + Rb^{-1}}{n}. $$  
    Since the second term on the right hand side above is less than $\eps/2$ for all $n$ large enough, we see that $B^c\cap A_n^c=\emptyset$ for all such $n$. 
\end{proof}

{
With this, we can prove the statement of Theorem \ref{thm:main} for the Gromov-weak topology (we will upgrade it to Gromov-Hausdorff weak in the next subsection).  
Recall that we consider 
 \[
  \mathcal T_{n, x} := (T, \frac1n d, \frac{1}{n^2}\nu, {\bf r}), \quad n \ge 0, x \in E,
\]
under law $\mathbb P_{\delta_x}( \cdot | N_n > 0)$, where $(T,d,\bf r)$ is the pointed metric space $([0, L)/\!\!\sim , d)$ defined from the genealogical tree $T$ of the MBP with its root $\bf r$, and $\nu$ is the push forward of Lebesgue measure on $[0, L)$ under the equivalence relation $\sim$. 

In other words, the law of the distances in $\mathcal{T}_{n,x}$ between $k$ points sampled from $(\nu/n^2)^{\otimes k}$ - normalised to be a probability measure - is exactly the law of 
$(D_n^h)$ from \eqref{eqn::matrices}. Moreover, the total mass of $|\nu/n^2|$ is by definition $L/n^2$ and under $\mathbb{P}_{\delta_x}(\cdot \mid N_n>0)$, $L/n^2$ converges to the length $\tau$ of a Brownian excursion run at speed $\sigma^2(f)$ and conditioned to reach height $1$, by Corollary \ref{cor:mgalefcltexcursion} {and Lemma \ref{lem::equiv_cond}.}

Thus, by the definition of Gromov-weak convergence, using Proposition \ref{prop::distancematrices} (which says that $D_n^h$ and $D_n^M$ are close) and Corollary \ref{cor:mgalefcltexcursion} (which says that $D_n^M$ and the distances between $k$ points in $T_{\bf e}$ chosen from $\mu_{\ee}^{\otimes k}$ are close, in law) we immediately obtain the following. }

\begin{prop}\label{prop:main}
Under Assumption \ref{A:main}, we have
\begin{equation}\label{eq:main}
\mathcal T_{n, x} \overset{d.}{\to} \mathcal T_{\bf e}, \quad n \to \infty
\end{equation}
with respect to the Gromov-weak topology. As in Theorem \ref{thm:main}, $T_{\bf e}$ is the Brownian CRT generated from a Brownian excursion run at speed $\sigma^2(f)$, from \eqref{eq:sf}, and conditioned to reach height one. 
\end{prop}

\subsubsection*{Step 3.}\label{subsec:LMBprop}
For $r>0$ and $w\in N_r$,  let $B_\delta(w,r)$ be the collection of all vertices with distance in $(T,d)$ at most $\delta$ from the point in $(T,d)$ corresponding to the particle with label $w$ at height $r$. Equivalently, this point is the `time' in $T=[0,L)$ given by $\inf\{t: (v_t,h_t)=(w,r)\}$. %{\color{red}Note that, since $w$ has an associated death time $d_w < \infty$, this ball will also depend on the ``age'' of the particle.} 
Further, let $|B_\delta(w,r)|$ be the total time in the depth-first exploration spent in $B_\delta(w,r)$. 

\begin{lemma}\label{lem:GLMB}
	For any $\delta>0$,
	$$\lim_{\eta\to 0} \limsup_{n \to \infty} \mathbb{P}_{\delta_x}(\exists w\,:\, |B_{4\delta n}(w,r)|<\eta n^2 \,|\, N_n>0)=0.$$
\end{lemma}

The proof of this lemma follows ideas in \cite{felix}.
\begin{proof}
	First observe that $\sup_{n\ge 1}\mathbb{P}_{\delta_x}(N_{K\delta n}>0|N_n>0)\to 0$ as $K\to \infty$, and $1/\mathbb{P}(N_n>0)\lesssim n$ by Theorem \ref{thm:survival}, so it suffices to show that for fixed $K$, we have 
	$$\lim_{\eta\to 0} \limsup_{n \to \infty} n\mathbb{P}_{\delta_x}(\exists w\,:\, |B_{4\delta n}(w)|<\eta n^2;\, N_n>0;\,N_{K\delta n}=0 )=0.$$
	For this we divide time into increments of length $\delta n$, and for a vertex $w\in \mathcal N_r$ we write $v_{w,r}$ for the ancestor of $w$ at time $s_r-1$ where $s_r=\sup\{k\delta n: k\delta n < r\}$. Consider the collection of trees $T_v^{2\delta n}$ rooted at some vertex $v=v_{w,r}$ with $w\in \mathcal N_r$, and cut off at height $2\delta n$. Then every ball $B_{4\delta n}(w,r)$ is contained in some such tree that survives for time at least $\delta n$. Thus, we have 
	\begin{align*}
		& n\mathbb{P}_{\delta_x}(\exists w\,:\, |B_{4\delta n}(w,r)|<\eta n^2;\, N_n>0;\,N_{K\delta n}=0 ) 
		\\
		& \le n\mathbb{P}_{\delta_x}(\cup_{k=1}^K \{ \text{some } v\in \mathcal N_{k\delta n} \text{ has }|T_v^{2\delta n}|<\eta n^2 \text{ and has descendants at generation }\delta n\}) \\
	& \le \sum_{k=1}^K \mathbb{E}_{\delta_x}\bigg[\sum_{v\in \mathcal N_{\delta k n}} n \mathbb{P}_{X_v(\delta k n)}(|T_v^{2\delta n}|<\eta n^2 \cap N_{\delta n}>0)\bigg] \\
	& \le  \sup_y n \mathbb{P}_y(|T_v^{2\delta n}|<\eta n^2 \cap N_{\delta n}>0) \sum_{k=1}^K \mathbb{E}[N_{\delta k n}].
		\end{align*}
		The second term in the final expression about is bounded by a constant (depending on $K$), while the first converges to $0$ uniformly in $n$ by Corollary \ref{cor:mgalefcltexcursion}. So by taking $K\to \infty$ and then $n\to \infty$ we prove our claim.
\end{proof}

\begin{proof}[Proof of Theorem \ref{thm:main}]
Lemma \ref{lem:GLMB} verifies the global lower mass bound condition in Lemma \ref{lem:GW-to-GHW},  so we can upgrade the Gromov-weak convergence of Proposition \ref{prop:main} to Gromov-Hausdorff-weak convergence. 
\end{proof}

\bibliographystyle{plain}
\bibliography{bibbp.bib}

\section*{Acknowledgements}
EH acknowledges support of the EPSRC grant MaThRad EP/W026899/2. EP is supported by UKRI Future Leaders 
Fellowship MR/W008513.
\end{document}